\documentclass[11pt, a4paper]{article}
\usepackage[T1]{fontenc}
\usepackage{a4wide, url}
\usepackage[english]{babel}
\usepackage{graphicx}
\usepackage{natbib}
\usepackage[textfont=sc,labelfont=bf]{caption}
\usepackage[blocks]{authblk}
\usepackage{setspace}
\usepackage{multirow}
\usepackage{amsthm}
\usepackage{amsmath}
\usepackage{amsfonts}
\usepackage{amssymb}
\usepackage{pgf} 
\usepackage{bm}
\usepackage{paralist}
\usepackage{multirow}

\newcommand{\mbf}{\mathbf}
\newcommand{\beq}{\begin{equation}}
\newcommand{\eeq}{\end{equation}}
\newcommand{\bea}{\begin{eqnarray}}
\newcommand{\eea}{\end{eqnarray}}
\newcommand{\bit}{\begin{itemize}}
\newcommand{\eit}{\end{itemize}}
\newcommand{\ben}{\begin{enumerate}} 
\newcommand{\een}{\end{enumerate}}
\newcommand{\bpm}{\begin{pmatrix}}
\newcommand{\epm}{\end{pmatrix}}
\newcommand{\bbm}{\begin{bmatrix}}
\newcommand{\ebm}{\end{bmatrix}}

\newtheorem{definition}{Definition}
\newtheorem{prop}{Proposition}
\newtheorem{rem}{Remark}
\DeclareMathAlphabet\mathbfcal{OMS}{cmsy}{b}{n}

\begin{document}

\title{{{\bf Dynamic Factor Models, Cointegration, and Error Correction Mechanisms}}}

\author{{\normalsize Matteo \textsc{Barigozzi}${}^{1}$\hskip 1cm \normalsize Marco \textsc{Lippi}${}^{2}$\hskip 1cm \normalsize Matteo \textsc{Luciani}}${}^{3}$}

\date{\small{\today}}

\maketitle

\vskip-.5cm
\begin{abstract} 
\noindent  The paper studies Non-Stationary Dynamic Factor Models such that the factors $\mbf F_t$ are $I(1)$ and singular, i.e.  $\mathbf F_t$ has dimension $r$ and is driven by a $q$-dimensional white noise, the common shocks,  with $q<r$. We show that 
$\mathbf F_t$ is driven by $r-c$ permanent shocks, where $c$ is the cointegration rank of $\mathbf F_t$, and $q-(r-c)<c$ transitory shocks, thus the same result as in the non-singular case for the permanent shocks but not for the transitory shocks.
Our main result is obtained by combining the classic Granger Representation Theorem  with recent results by Anderson and Deistler on singular stochastic vectors: if $(1-L)\mathbf F_t$ is singular and  has   {\it rational} spectral density then, 
for generic values of the parameters, $\mbf F_t$  has an autoregressive representation with a {\it finite-degree} matrix polynomial fulfilling the restrictions of  a Vector  Error Correction Mechanism  with $c$ error terms.  This result is the basis for consistent estimation of  Non-Stationary Dynamic Factor Models. The relationship between cointegration of the factors and cointegration of the observable variables  is also discussed.\\

\vspace{3mm}


\end{abstract}

\thispagestyle{empty}

 \footnotetext[1]{m.barigozzi@lse.ac.uk -- London School of Economics and Political Science, UK.} 

\footnotetext[2]{ml@lippi.ws -- Einaudi Institute for Economics and Finance, Roma, Italy.}

\footnotetext[3]{matteo.luciani@frb.gov -- Federal Reserve Board of Governors, Washington DC, USA.\\

\noindent  Massimo Franchi, and Rocco Mosconi read previous versions of the paper and gave suggestions for 
improvements. We also thank the  participants to the Workshop on Estimation and Inference Theory for Cointegrated Processes in the State Space Representation,   Technische Universit\" at Dortmund, January 2016.
Of course  we are responsible for any remaining errors. 
The  paper  was  written  while  Matteo  Luciani  was \textit{charg\'e de recherches} F.R.S.-F.N.R.S., whose financial support he   gratefully acknowledges. 
The views expressed here are those of the authors and do not necessarily reflect those of the Board of Governors or the Federal Reserve System.
 }

\newpage

\section{Introduction}\label{introduction}

In the last fifteen years Large-Dimensional Dynamic Factor Models  (DFM) have become increasingly popular in the economic and the econometric  literature and they are nowadays commonly used by policy institutions. The success of these models have fostered a large effort by the academic community in studying the theoretical properties of these models. Despite few exceptions \citep{baing04,bai04,penaponcela04}, the literature on DFMs has studied only the case in which the data are stationary. In this paper, we move away from the stationary setting and we provide representation results, which constitute the basis upon which we can build a non-stationary DFM, whose estimation is studied in a companion paper \citep{BLLirf}.

Factor models are based on the idea that  all the variables in an economic system are driven  by  a few common (macroeconomic) shocks, their residual dynamics being explained by idiosyncratic components  which may result from 
  measurement errors and sectoral or regional shocks. Formally, each variable in the $n$-dimensional  dataset $x_{it},\ i=1,2,\ldots,n$ can be decomposed into the sum of a common component $\chi_{it}$,  and an idiosyncratic component $\epsilon_{it}$: $x_{it} =\chi_{it}+ \epsilon_{it}$  \citep{FHLR00,fornilippi01,stockwatson02JASA,stockwatson02JBES}.  In the standard version of the DFM, which is adopted here,  the  
 common components are   linear combinations of an $r$-dimensional vector of  common factors $\mathbf F_t=(F_{1t}\ F_{2t}\ \cdots\ F_{rt})'$,  
\begin{equation}\label{pennetta}\chi_{it}=\lambda_{i1}F_{1t}+\lambda_{i2}F_{2t}+\cdots + \lambda_{ir}F_{rt}=\bm\lambda_i \mbf F_t.\end{equation}
The vector $\mathbf F_t$ is dynamically  driven by the $q$-dimensional  non-singular white-noise vector\footnote{Usually orthonormality is assumed. This is convenient but  not necessary in the present paper.} $\mathbf u_t=(u_{1t}\ u_{2t}\ \cdots\ u_{qt})'$,  the common shocks: 
\begin{equation}\label{vinci}\mbf F_t=\mbf U(L) \mbf u_t,\end{equation} 
where $\mathbf U(L)$ is an $r\times q$ matrix, \citep{stockwatson05,baing07,FGLR09}.    
The dimension $n$ of the dataset is assumed to be large as compared to $r$ and $q$, which  are independent of $n$, with $q\leq r$.   More precisely,  all assumptions and results are formulated assuming  that both $T$, the number of observations for each $x_{it}$, and  $n$, the number of variables,  tend to infinity.

The assumption  that  the vector $\mathbf F_t$ is singular, i.e. $r>q$, has received  sound  empirical support in a number of papers analysing macroeconomic databases, see, for example,
\cite{GRS04}, \cite{amengualwatson07}, \cite{fornigambettiJME}, and \cite{smokinggun} for the US and \cite{BCL} for the Euro area.
Such results  can be easily understood  observing that the static equation
  \eqref{pennetta}  is  just  a convenient representation  derived from a ``primitive'' set of  dynamic equations linking the common components $\chi_{it}$ to the common shocks $\mathbf u_t$.   As a simple example,  suppose that the variables $x_{it}$ belong to a macroeconomic dataset and are driven by
  a common  one-dimensional cyclical  process  $ f_t$, such that $(1-\alpha L) f_t = u_t$, where $u_t$ is scalar white noise, and that   the variables $x_{it}$ load 
  $f_t$ dynamically: 
  \begin{equation}\label{vwilliams}x_{it} = a_{i0} f_t + a_{i1} f_{t-1}+\epsilon_{it}.\end{equation}
 In this case representation \eqref{pennetta} is obtained by setting $F_{1t}=f_t$, $F_{2t}=f_{t-1}$, $\lambda_{i1} = a_{i0}$, $\lambda _{i2} = a_{i1}$, while 
  equation  \eqref{vinci} takes the form
  $$ \begin{pmatrix} F_{1t}\\ F_{2t} \end{pmatrix} = \begin{pmatrix} (1-\alpha L)^{-1} \\ (1-\alpha L)^{-1} L \end{pmatrix} u_t,$$
  so that $r=2$,  $q=1$ and 
 the  dynamic equation \eqref{vwilliams} is replaced by the static representation $x_{it} = \lambda_{i1} F_{1t}+\lambda_{i2} F_{2t} +\epsilon _{it}$. For a general analysis of the relationship between representation \eqref{pennetta} and  ``deeper''  dynamic representations like \eqref{vwilliams}, see
    e.g.   \cite{stockwatson05}, \cite{FGLR09}, see also Section \ref{tristru} below.  Singularity of $\mathbf F_t$, i.e. $r>q$, will be assumed throughout the present paper.

If the factors $\mbf F_t$ and the idiosyncratic terms are stationary, and hence the data $x_{it}$  are stationary as well, the factors  $\mathbf F_t$ and the loadings  $\pmb \lambda_i$  can be consistently  estimated using the first $r$ principal components \citep{stockwatson02JASA,stockwatson02JBES}. The common shocks $\mbf u_t$ and the  function 
 $\mbf U(L)$  can  then be  estimated  by a   singular   VAR for  $\mbf F_t$.  Lastly,
 identification restrictions can be applied to the shocks $\mathbf u_t$ and the function $\mathbf U(L)$ to obtain structural common shocks and impulse response functions, see \cite{stockwatson05}, \cite{FGLR09}.
 
When the factors are $I(1)$, so that the observables variables $x_{it}$ are $I(1)$ as well, equation \eqref{pennetta} remains unchanged while \eqref{vinci} is replaced by
  \begin{equation}\label{poverahalep} (1-L) \mathbf F_t = \mathbf U(L) \mathbf u_t.\end{equation} 
If the idiosyncratic components are stationary the factors $\mbf F_t$ are directly estimated as the principal components of the $I(1)$ variables $x_{it}$ \citep{bai04,penaponcela04}. However, if the idiosyncratic components are non-stationary the principal components of the stationary series $(1 - L)x_{it}$ provide an estimate of the differenced factors $(1 - L)\mbf F_t$ and of the loadings $\bm\lambda_i$, and the factors $\mbf F_t$ can then be recovered by integration, see e.g. \citet{baing04}.
    
The main difference with respect to the stationary case arises with the  estimation of $\mathbf u_t$ and $\mathbf U(L)$,  or  structural common shocks and impulse-response functions.  This requires the estimation of an autoregressive model for the $I(1)$  factors $\mathbf F_t$, which, due to singularity of  $\mathbf F_t$,  are quite obviously cointegrated (the  spectral density of $(1-L)\mathbf F_t$ is singular at all frequencies and therefore at frequency zero).   Here we study  the autoregressive representations  of thesingular cointegrated vector $\mathbf F_t$, while estimation  is studied in  \cite{BLLirf}.\footnote{To our knowledge, the present paper is the first  to  study cointegration and Error Correction representations    for the   singular factors of 
$I(1)$  Dynamic Factor Models.  An Error Correction model in the DFM framework  is studied in \cite{schifez,ancorapiuschifez}. However, their 
focus  is on the relationship between the observable variables and the factors. Their Error Correction term is a linear combination  of  the variables $x_{it}$ and the factors $\mathbf F_t$, which is stationary if the idiosyncratic components are stationary (so that the $x$'s and the factors are cointegrated).  Because of this and other important differences their results  are not directly comparable to those in the present paper.}

The paper is organized as follows. In Section \ref{sop} we recall recent  results for singular stochastic vectors 
with rational spectral density, see \citet{andersondeistler08,andersondeistler0free}, and we
  discuss  cointegration and  the  cointegration rank for $I(1)$ singular stochastic vectors: $c$, the cointegration  rank, is equal to 
  $r-q$,    the minimum due to singularity, plus $d$, with $0\leq d<q$.
  
   In Section \ref{hauptbanhof}  we obtain the permanent-transitory shock representation in the singular case:   $\mathbf F_t$ is driven by 
  $r-c=q-d$ permanent and $d=c-(r-q)$ transitory shocks,
  the same result as in the non-singular case for the 
  permanent shocks but not for the transitory.  Then we prove our main results. Assuming rational spectral density for the 
  vector $(1-L) \mathbf F_t$ and therefore that  the entries of $\mathbf U(L)$ in \eqref{poverahalep} are rational fuctions of $L$, then
{\it for generic values of the parameters  of the matrix $\mathbf U(L)$},   $\mathbf F_t$  has an autoregressive  representation  fulfilling the restrictions of a  Vector Error Correction Mechanism (VECM) with $c$ error terms:
\begin{equation}\label{poverawilliams}\mathbf A(L) \mathbf F_t= \mathbf A^*(L) (1-L) \mathbf F_t + \pmb \alpha \pmb \beta' \mathbf F_{t-1} = \mathbf h +\mathbf R \mathbf u_t,\end{equation}
where $\pmb \alpha $ and $\pmb \beta $ are both $r\times c$ and full rank,  $\mathbf R$ is $r\times q$, $\mathbf A(L)$ and  $\mathbf A^*(L)$ are {\it finite-degree} matrix polynomials. These results are obtained by combining the Granger Representation Theorem \citep{englegranger1987} with  Anderson and Deistler's results.  
 
Section \ref{hauptbanhof} also contains  an  exercise carried on  with simulated  singular $I(1)$ vectors. We compare the results obtained by estimating an unrestricted VAR in the  levels and  a VECM. Though limited to a simple example, the results   confirm  what has been found for non-singular vectors, that is under cointegration the long-run features of impulse-response  functions are better estimated using a VECM rather than an unrestricted VAR in the levels \citep{phillips98}.

In Section \ref{cointegrationDFM} we analyse  cointegration of the observable variables $x_{it}$. Our results on cointegration of the factors $\mathbf F_t$ have the obvious implication that {$p$-dimensional} subvectors of the $n$-dimensional  common-component vector  $\pmb \chi_t$, with $p>q-d$, are cointegrated.  Stationarity of the idiosyncratic components would imply that  all $p$-dimensional subvectors of the $n$-dimensional dataset $\mathbf x_t$ are cointegrated. For example, if $q=3$ and $d=1$, then  all
$3$-dimensional subvectors in the dataset are cointegrated,   a kind of regularity that we do not  observe in actual large macroeconomic datasets. This suggests that an estimation strategy robust to the assumption that the idiosyncratic components are $I(1)$ (some of the variables $\epsilon _{it}$ are $I(1)$) has to be preferred \citep[for this aspect we refer to][]{BLLirf}. Section \ref{conclusion}  concludes.  Some long  proofs, a discussion of some non-uniqueness problems arising with singularity and details on the simulations are collected in  the Appendix.

\section{Stationary and non-stationary Dynamic Factor  Models } \label{sop}
\subsection{The  Factors and Idiosyncratic Components are Stationary}\label{sop1}Consider the Dynamic Factor Model
\begin{equation} \label{unouno} 
\mbf x_t=\bm \chi_t +\bm\epsilon_t,\ \ \ \ \bm\chi_t = \bm \Lambda \mbf F_t,
\end{equation}
where: (1)~the observables $\mbf x_t$, the common components $\bm \chi_t$, and the idiosyncratic components $\bm \epsilon_t$ are $n$-dimensional vectors, (2)~$\mbf F_t$ is an $r$-dimensional vector of common factors, with $r$ independent of $n$, (3)~$\bm \Lambda$ is an $n\times r$ matrix, (4)~$\mbf F_t$ is driven by a $q$-dimensional zero-mean white noise vector process $\mbf u_t$, the common shocks, with $q  < r$, (5)~$\epsilon_{it}$ and $u_{js}$ are orthogonal for all $i=1,2,\ldots,n$, $j=1,2,\ldots,q$, $t,s\in \mathbb Z$.  Other assumptions concerning  the asymptotic  properties of the model, for $n\to\infty$, will not be used here and are therefore not 
reported (see the literature mentioned in the Introduction).    The results in the present and the next section, though stated for 
the  vector of the  factors $\mathbf F_t$ and the common shocks $\mathbf u_t$,  hold for any 
singular stochastic vector  under the assumptions specified below.

As observed in the Introduction, with some exceptions, the theory of model \eqref{unouno} has been developed under the assumption that $\mbf x_t$, $\bm \chi_t$, $\bm \epsilon_t$ and $\mbf F_t$ are stationary.  In addition to \eqref{unouno}, it is often assumed that $\mbf F_t$ has a reduced-rank VAR representation:
\beq\label{modena} 
\mbf A(L) \mbf F_t= \mbf R \mbf u_t,
\eeq
where $\mbf A(L)$ is a finite-degree $r\times r$ matrix polynomial and $\mbf R$ is $r\times q$. Moreover, it is well known that  $\bm \Lambda$ and $\mbf F_t$ can be estimated by principal components, while an estimate of $\mbf A(L)$, $\mbf R$ and $\mbf u_t$ can be obtained by standard techniques. Inversion of $\mbf A(L)$ provides an estimate of the  impulse-response functions of the observables to the common shocks:
$$\mbf x_t = \bm \Lambda\mbf A(L) ^{-1} \mbf R \mbf u_t +\bm \epsilon_t.$$
Structural shocks and structural impulse-response functions can then be obtained, respectively, as $\mbf w_t = \mbf Q \mbf u_t$ and $ \bm \Lambda\mbf A(L) ^{-1} \mbf R \mbf Q^{-1}$, where the $q\times q$ matrix $\mbf Q$ is determined in the same way as in Structural VARs  \citep{stockwatson05,FGLR09}.

The VAR representation \eqref{modena} has a standard motivation as an approximation to an infinite autoregression with exponentially declining coefficients. However, as stated above, $\mbf F_t$ has reduced rank. Under reduced rank and rational spectral density for $\mbf F_t$, \cite{andersondeistler08,andersondeistler0free} prove that {\it generically} $\mbf F_t$ has a finite-degree autoregressive representation, so that no approximation argument is needed to motivate (\ref{modena}). A formal statement of this result requires the following definitions.

\begin{definition}\label{def1} \textbf{(Rational reduced-rank family)}
Assume that $r>q>0$ and let $\cal G$ be a set of ordered couples $(\mbf S(L),\mbf C(L))$, where:
\begin{compactenum}[(i)]
	\item $\mbf C(L)$ is an $r\times q$ polynomial matrix of degree $s_1\geq 0$.
	\item $\mbf S(L)$ is an $r\times r$ polynomial matrix of degree $s_2\geq 0$. $\mbf S(0) =\mbf I_r$.
	\item Denoting by $\mathbf p$  the vector containing the $\lambda = rq(s_1+1)+r^2s_2$ coefficients of the entries of $\mbf C(L)$ and $\mbf S(L)$, we assume that $\mathbf p\in \Pi$, where $\Pi$ is an open subset of $\mathbb R^\lambda $, and that for $\mathbf p\in \Pi$, if $\det (\mathbf S(z))=0$, then $|z|>1$.
\end{compactenum}

\noindent We say that the family of weakly stationary stochastic processes 
\begin{equation}\label{kamikaze}\mbf F_t =\mbf S(L)^{-1}\mbf C(L) \mbf u_t,\end{equation}
where $\mbf u_t$ is a $q$-dimensional white noise with non-singular variance-covariance matrix and
$(\mbf S(L),\mbf C(L))$ belongs to $\cal G$, is a \it rational reduced-rank family. 
\end{definition}
The notation  $\mathbf F^p_t$, $\mathbf C^p(L)$, etc.,  though more rigorous, would  be heavy and not really necessary. We use it only once in the proof of Proposition \ref{montidc}.

Note that (\ref{kamikaze}) is the unique stationary solution of the ARMA equation 
\begin{equation} \mbf S(L) \mbf F_t = \mbf C(L) \mbf u_t.\label{sushi}\end{equation}

\begin{definition}\label{defff2}\textbf{(Genericity)} 
Suppose that a statement $ Q (\mathbf p)$ depends on $\mathbf p\in {\cal A}$, where $\cal A$ is an open subset of $\mathbb R^\lambda$. Then $Q(\mathbf p)$ holds generically in $\cal A$ if the subset $\cal N$ of $\cal A$ where it does not hold is nowhere dense in $\cal A$, i.e. the closure of $\cal N$ has no internal points.
\end{definition}

\begin{prop}\label{propp1} \textbf{(Anderson and Deistler)} (I)~Suppose that $\mbf V(L)$ is an $r\times q$  matrix whose entries are rational functions of $L$, with $r>q$. If $\mbf V(L)$ is zeroless, i.e. has rank $q$ for all complex numbers $z$, then $\mbf V(L)$ has a finite-degree  stable left inverse, i.e. there exists a finite-degree polynomial $r\times r$ matrix $\mbf W(L)$, such that (a)~$\det (\mbf W(z))=0$ implies $|z|>1$, (b)~$\mbf W(L) \mbf V(L) = \mbf V(0)$. (II)~Let $\mbf F_t = \mbf S(L) ^{-1} \mbf C(L) \mbf u_t$ be a rational reduced-rank family with parameter set $\Pi$. For generic values of the parameters  $\mathbf p\in \Pi$, $\mathbf S(z)^{-1}\mbf C(z)$ is zeroless.  In particular, generically $\mathbf S(1)^{-1} \mathbf C(1)$  and $\mathbf S(0)^{-1} \mathbf C(0)=\mathbf C(0)$ have full rank $q$.
\end{prop}

For statement (I)  see \cite{andersondeistlersingular}, Theorem 3.   Statement (II) is a modified version of their Theorem 2. They obtain genericity with respect  to the parameters of the state-space representation  of \eqref{sushi}, whereas  in statement (II) we refer to the original parameters 
 of the matrix polynomials in (\ref{sushi}) (see   \cite{FHLZ2014} for a proof). Another 
 version of Anderson and Deistler's  Theorem  2 is Proposition \ref{montidc} in the present paper. 
  Both statement (I) and our Proposition \ref{montidc} are crucial for the proof of our main results in Propositions \ref{grtrr} and \ref{montidc230}.

\subsection{Non-Stationary factors}\label{sop2} 
Suppose now that  the vector $\mathbf F_t$ is non-stationary while $(1-L) \mathbf F_t$ is stationary. The common practice in this case   consists in reducing   the data $x_{it}$  to stationarity by taking first differences and estimating the differenced  factors $(1-L) \mathbf F_t$  by means of the principal components of the variables $(1-L) x_{it}$.  Then usually   impulse-response functions are obtained by estimating a VAR for $(1-L) \mathbf F_t$. Of course this implies  possible misspecification if $\mathbf F_t$ is cointegrated, which    is always the case when $\mathbf F_t$ is singular.   

To analyse  cointegration and the autoregressive representations of the singular non-statio\-na\-ry  vector   $\mathbf F_t$ 
let us firstly recall the definitions of $I(0)$, $I(1)$ and cointegrated vectors.  

In the present paper we only consider stochastic vectors that are either weakly stationary  with a rational spectral density matrix or 
such that their first difference is  weakly stationary with rational spectral density.  
 Assume that the $n$-dimensional  vector $\mathbf y_t$ is weakly stationary with rational spectral density, denoted by $\pmb \Sigma_{ y}(\theta)$.    The matrix $\pmb \Sigma_{ y}(\theta)$
 has constant rank $\rho$, with $\rho\leq n$, i.e. has the same rank $\rho$ for $\theta$ almost everywhere in $ [-\pi,\ \pi]  $.  
 We say that $\rho$ is the rank of $ \mathbf y_t$.
  Moreover, 
 $ \mathbf y_t$ has  moving average representations
 \begin{equation}\label{timedomain}\mathbf y_t =\mathbf V(L) \mathbf v_t,\end{equation}
where $\mathbf v_t$ is a non singular $\rho$-dimensional white noise,  $\mathbf V(L)$  is an $n\times \rho$ matrix whose entries are rational functions of $L$ with no poles of modulus less or equal to unity. 
If  ${\rm rank}(\mathbf V(z))=\rho $ for  $|z|<1$,  then $\mathbf v_t$ belongs to the space spanned by 
$\mathbf y_\tau$, with $\tau \leq t$, and representation \eqref{timedomain}, as well as $\mathbf v_t$, is called {\it fundamental} (see \cite{rozanov}, pp. 43--7)\footnote{When $n=p$, the condition that ${\rm rank}(\mathbf V(z))=p $ for  $|z|<1$ becomes 
$\det(\mathbf V(z))\neq 0$ for $|z|<1$.}.

Let $\mathbf z_t $ be an $r$-dimensional  weakly stationary with rational spectral density and assume that $\mathbf z_t\in L_2(\Omega,{\cal F},P)$.   Consider the difference equation
\begin{equation}\label{familyday} (1-L)\pmb \zeta_t =\mathbf  z_t,\end{equation}
in the unknown process $\pmb\zeta_t$. A solution of \eqref{familyday} is 
\begin{equation}\label{martinlutero}\mathbf {\tilde y}_t =\begin{cases} \hbox{$\mathbf z_1+\mathbf z_2+\cdots +\mathbf z_t$, for $t>0$}\\ \hbox{$\mathbf 0$, for $t=0$}\\ \hbox{$-(\mathbf z_0+\mathbf z_{-1}\cdots +\mathbf z_{t+1})$, for $t>0$}.\\ \end{cases} \end{equation}
All the solutions of \eqref{familyday}  are $\mathbf y_t=\mathbf {\tilde y}_t +\mathbf W$, where $\mathbf W$ is any $r$-dimensional  stochastic vector belonging to $L_2(\Omega,{\cal F},P)$, that is the particolar solution $\mathbf {\tilde y}_t$ plus any solution of $(1-L)\pmb \zeta_t =\mathbf 0$.

 \begin{definition} \label{iuno}  \textbf {(I(0), I(1) and cointegrated vectors)}

\noindent  \textbf {I(0).} The $n$-dimensional vector stochastic process  $\mathbf y_t$ is  $I(0)$ if it is weakly  stationary with  rational spectral density
 $\pmb \Sigma_y(\theta)$ and  $\pmb \Sigma_y(0)\neq \mathbf 0$. 
 
\noindent \textbf{Integrated process of order 1, I(1).} The $n$-dimensional vector  stochastic  process  $\mathbf y_t$ is  $I(1)$ if  there exists an $n$-dimensional process $\mathbf z_t$, weakly stationary with rational spectral density, such that $\mathbf y_t$ is a solution of  the equation $(1-L)\pmb \zeta _t=\mathbf z_t$.   The rank of $\mathbf y_t$ is defined as the rank of $\mathbf z_t$.

\noindent \textbf{Cointegration.}    Assume that the $n$-dimensional stochastic vector $\mathbf y_t$ is $I(1)$ and denote by
$\pmb \Sigma_{\Delta y}(\theta)$ the spectral density of $(1-L)\mathbf y_t$. The   vector  $\mathbf y_t$  is cointegrated with cointegration rank $c$,  with $0<c<n$, if  {\rm rank}$(\pmb\Sigma_{\Delta y}(0))=n-c$.  If $\rho$ is the rank of $\mathbf y_t$,  then $c\geq n-\rho$.

%
 \end{definition}
 
 \goodbreak
 Some  comments are in order.
  
  \begin{enumerate}
 \item[C1.]  If $\mathbf y_t$ has  representation \eqref{timedomain},  
 because $\pmb \Sigma_{ y}(0) = (2\pi)^{-1} \mathbf V(1) \pmb \Gamma_v\mathbf V(1)'$,  where $\pmb \Gamma _v$ is the covariance matrix of $\mathbf v_t$, $\mathbf y_t$ is $I(0)$ if and only if $\mathbf V(1) \neq \mathbf 0$. 
 
\item [C2.] Under the parameterization in Definition \ref{def1},  for generic values of the parameters in $\Pi$, we have  ${\rm rank}\left (\mathbf S(1)^{-1}\mathbf C(1)  \right )=q$ (see proposition \ref{propp1}(II)), so that  
$${\rm rank}\left (\pmb \Sigma_F(0)\right ) = {\rm rank}\left ((2\pi)^{-1} \mathbf S(1)^{-1} \mathbf C(1) \pmb \Gamma_ u \mathbf C(1)' [\mathbf S(1)']^{-1}\right )=q, $$
where $\pmb \Gamma_u$ is the covariance matrix of $\mathbf u_t$. Thus, generically,  $\mathbf F_t$   has rank $q$ and is $I(0)$.


\item[C3.] Assume that $\mathbf y_t$ is such that $(1-L)\mathbf y_t$ is weakly stationary with rational spectral density 
and that 
\begin{equation}\label{dayfamily}(1-L) \mathbf y_t = \mathbf V(L) \mathbf v_t\end{equation}
is  one of its moving average representations.  The process 
$\mathbf y_t$ is $I(1)$ if and only if $\mathbf V(1)\neq \mathbf 0$.    Thus our definitions of $I(0)$ and $I(1) $ processes
are equivalent  to Definitions 3.2, and 3.3    in \cite{Johansen95}, p.~35, with two minor differences: our assumption of rational spectral density and  the time span of stochastic processes, $t\in \mathbb Z$ in the present paper, $t=0,1,\ldots $ in Johansen's book.
%

\item[C4.]  Note that $\mathbf y_t$ can be $I(1)$ 
even though some  of its 
 coordinate processes  are $I(0)$. 

\item[C5.]  
 If $\mathbf y_t$ is $I(1)$ and cointegrated with cointegration rank $c$, there exist 
 $c$  linearly independent $n\times 1$ vectors $\mathbf c_j$, $j=1,\ldots ,c$, such that the spectral density of  $\mathbf c_j '(1-L) \mathbf  y _t$ vanishes at frequency zero.  The vectors $\mathbf c_j$ are called cointegration vectors.    Of course 
 a set of cointegration vectors $\mathbf c_j$, $j=1,\ldots,  c$, can be replaced by the set $\mathbf d_j$, $j=1,\ldots,c$, where 
 the vectors $\mathbf d_j$ are $c$  independent linear combinations  of the vectors $\mathbf c_j$. 
 
 \item[C6.]      In the literature on integrated and cointegrated  vectors,   the expression  ``$I(1)$ process'' is often ambiguous,
 sometimes  it refers to  a specific process $\mathbf y_t$ which solves $(1-L)\pmb\zeta_t=\mathbf z_t$, sometimes to  the whole class  $\mathbf y_t=\mathbf {\tilde y}_t +\mathbf W$.    This minor abuse of language is   very convenient 
 and usually  does not  cause misunderstandings, see comment C7 below, Section \ref{ECM}, Proposition \ref{grtrr} in particular.

 \item[C7.]  If $ \mathbf y_t$ is $I(1)$, cointegrated   and has representation \eqref{dayfamily},
the cointegration rank of $\mathbf y_t$ is  $c$ if and only if the rank of $\mathbf V(1)$ is $n-c$. Moreover $\mathbf c$ is a cointegration 
vector for $\mathbf y_t$  if and only if $\mathbf c'\mathbf V(1)=\mathbf 0$.   In \ref{A} we show that  if  $\mathbf c$ is a cointegration vector for $\mathbf y_t$, then  $\mathbf y_t$ can be determined (that is, a member of the class  containing $\mathbf y_t$ can be determined)  such that $\mathbf c' \mathbf y_t$ is weakly stationary with rational spectral density.   Thus our definition  of cointegration is equivalent to  that in
\cite{Johansen95}, p.~37.

\item[C8.] Let $\mathbf y_t$ be $n$-dimensional, $I(1)$ with rank $\rho<n$. The spectral density $\pmb \Sigma_{\Delta y}(\theta)$ has rank $\rho$  almost everywhere in $[-\pi,\ \pi]$, so that there exist  vectors $\mathbf d_j(e^{-i\theta})$, $j=1,2,\ldots,n-\rho$, such that $\mathbf d_j(e^{-i\theta}) \pmb \Sigma _{\Delta y}(\theta) =\mathbf 0$ for all $\theta$.  In the time domain,
$ \mathbf d_j(L) (1-L) \mathbf y_t=0$.     Such  exact linear dynamic relationships between the coordinates of $\mathbf y_t$  (not possible  when $\rho=n$)
 are closely linked to the non-uniqueness of autoregressive representations   for singular vectors  (see below in this section and  \ref{nu}).

\item[C9.]  Definition \ref{iuno}, Cointegration, does not rule out that an eigenvector of  $\pmb \Sigma _{\Delta y}(\theta)$  constant, i.e. does not depend on $\theta$.  If $\mathbf d$ is  such an eigenvector, $\mathbf d(1-L) \mathbf y_t=0$, which is a
degenerate case of cointegration (again, not possible when $\rho=n$). More on this in comment  C3 on Definition \ref{def2}.


\end{enumerate}

%

Now, suppose that $\mbf F_t$  is $I(1)$,
 that 
 $$ (1-L) \mbf F_t = \mathbf S(L)^{-1}\mbf C(L) \mbf u_t=\mathbf U(L) \mathbf u_t$$
and that the rank of the $r\times q$ matrix $\mathbf U(z)$ is $q$ for almost all $z\in \mathbb C$, i.e.  that the rank of $\mathbf F_t$ is $q$.  Then:

\begin{enumerate}[(i)]
	\item   Obviously, as already observed in the definition of cointegration,  $\mathbf F_t$ has at least $r-q$ cointegration vectors: $c\geq r-q$. 
	\item   If we assume that the couple $(\mathbf S(L),\mathbf C(L)$ is parameterized as in Definition \ref{def1}, then,
by  Proposition \ref{propp1}(II),  we can  argue that 
	 generically $\mathbf S(z)^{-1}\mbf C(z)$ has full  rank $q$ for all $z$ and therefore the cointegration rank of $\mbf F_t$ is generically $r-q$. However, the rank of $\mathbf S(z)^{-1}\mbf C(z)$ at $z=1$ has a special interpretation
	 as  the number of  long-run equilibrium relationships between the processes $F_{ft}$.  Such number usually has a theoretical or behavioral motivation, so that it cannot be modified by any genericity argument.   As a consequence  we 
adopt a different parameterization  for  families of $I(1)$ vectors  in which the cointegration rank $c$   is  fixed, with $r>c\geq r-q$, see Definition \ref{def2}.
	\item If the family has  $c=r-q$, then generically  rank$(\mathbf S(1)^{-1}\mbf C(1))=q$ and   Proposition \ref{propp1}(I) can be applied. In spite of cointegration, generically $(1-L)\mbf F_t$ has a finite-degree autoregressive representation
\beq\label{compagnoalfano} \mbf A(L) (1-L)\mbf F_t = \mbf C(0) \mbf u_t.\eeq
	\item If  the family has $c>r-q$, then no autoregressive representation exists for $(1-L)\mbf F_t$, finite or infinite. However, we prove that, if $c$ is equal or greater than $r-q$, generically $\mbf F_t$ has a  representation as a VECM.
\beq\label{compagnocicchitto}
\mbf A(L) \mbf F_t=\mbf A^* (L) (1-L) \mbf F_t +\mbf A(1) \mbf F_{t-1} =\mbf h +\mbf C(0) \mbf u_t,
\eeq
where the rank of $\mbf A(1)$ is $c$, and $\mbf A(L)$ and $\mbf A^*(L)$ are \it finite-degree\rm\ matrix polynomials. 
	\item  In the singular case the  autoregressive representation of $\mathbf F_t$ is in general  not  unique (as it is the case with full-rank vectors). For example, when $c=r-q$, (\ref{compagnoalfano}) and (\ref{compagnocicchitto}) are two different  autoregressive representations for $\mbf F_t$, the first with no error terms, the second with $r-q$ error terms.  However, as we show in  \ref{nu},
	different   autoregressive representations of $\mathbf F_t$ produce the same impulse-response functions.    Proposition \ref{grtrr} in the next section proves the existence 
	of representation \eqref{compagnocicchitto}, which has the maximum number of  error terms.
\end{enumerate}

The existence of representation (\ref{compagnocicchitto}) for reduced-rank $I(1)$ vectors, where $\mbf A^*(L)$ is of finite degree, is our main result. Its proof combines the Granger Representation Theorem  with  the results summarized in Proposition \ref{propp1}.

\subsection{``Trivial'' and ``primitive'' cointegration vectors  of $\mathbf F_t$}\label{tristru}
Denote by $d$ the number of cointegrating vectors exceeding the minimum $r-q$, so that $c=r-q+d$. Of course $q>d\geq 0$.
 There is an interesting class of vectors  $\mathbf F_t$ for which 
it is possible  to distinguish between  $r-q$ cointegrating vectors that  merely arise in the construction of $\mathbf F_t$, 
and $d$ additional cointegrating  vectors   with a possible   structural interpretation.

 Let $\mathbf f_t$ be a non-singular $q$-dimensional vector of   primitive factors  $\mathbf f_t$ with representation 
 \begin{equation}\label{primitivo} (1-L)\mathbf f_t= \mathbf U_f (L) \mathbf u_t,\end{equation}
 and assume that the variables $x_{it}$ load $\mathbf f_t$ and its lags up to the $p$-th:
 \begin{equation}\label{primitivo1} x_{it} = \mathbf a_{i0} \mathbf f_t +\mathbf a_{i1}
\mathbf f_{t-1} +\cdots + \mathbf  a_{ip} \mathbf f_{t-p}+\epsilon _{it}.\end{equation}
This is a generalization of
 the  example used in the Introduction to motivate the singularity of the vector $\mathbf F_t$, see equation \eqref{vwilliams}.
Model \eqref{primitivo}--\eqref{primitivo1}  is transformed into the standard  form \eqref{pennetta} by introducing the 
$r$ dimensional vector
$$\mathbf F_t =\begin{pmatrix} \mathbf f_t'& \mathbf f_{t-1}'&\cdots & \mathbf f_{t-p}'\end{pmatrix}',$$
where $r=q(p+1)$.  We have 
\begin{equation}\label{bellamattina}\mathbf x_t = \pmb \Lambda \mathbf F_t +\pmb \epsilon_t,\ \ \ \  \pmb \Lambda=\begin{pmatrix}  \mathbf a_{10} & \mathbf a_{11} & \cdots & \mathbf a_{1p}\\     \mathbf a_{20} & \mathbf a_{21} & \cdots & \mathbf a_{2p} \\ \vdots \\  \mathbf a_{n0} & \mathbf a_{n1} & \cdots & \mathbf a_{np} \end{pmatrix}
, \ \ \ \ (1-L) \mathbf  F_t = \begin{pmatrix} \mathbf U_f(L) \\ L \mathbf U_f(L) \\ \vdots \\L^p \mathbf U_f (L)\\ \end{pmatrix} \mathbf u_t = \mathbf U(L) \mathbf u_t.
\end{equation}
It is immediately seen that 
$\mathbf F_t$ has $r-q=qp$    cointegrationg vectors 
$$\mathbf t^{h,k}=\left ( t^{h,k}_1\  t^{h,k}_2\ \cdots\  t_r^{h,k}\right ),$$ 
$h=1,\ldots,q$, $k=1,\ldots,p$, where
$$  t^{h,k}_j =\begin{cases} \hbox{$1$ if $j=h$}\\ \hbox{$-1$ if $j=kq+h$}\\ \hbox{$0$ otherwise.}\end{cases} $$
These  can be called  trivial cointegration vectors, as they merely result from  the  construction of $\mathbf F_t$ by  stacking the 
vectors $\mathbf f_{t-k}$, $k=0,\ldots,p,$.  

On the other hand, if $\mathbf f_t$ is cointegrated,
with cointegration rank $c_f$, $q>c_f> 0$, and cointegrating vectors $\mathbf s^m$, $m=1,\ldots,c_f$, 
then the cointegration rank of $\mathbf F_t$ is $c=r-q+c_f$  with the additional $c_f$ cointegrating vectors 
obtained by augmenting each $\mathbf s^m$ with $qp$ zeros.  Thus cointegration of the  primitive
factors $\mathbf f_t$   naturally translates into cointegration of  $\mathbf F_t$.

Two observations are in order.  Firstly,
 the above distinction between primitive and trivial cointegration  has heuristic 
 interest but is limited to representation \eqref{bellamattina}, the latter being 
 only one among infinitely  many equivalent standard  representations  of
\eqref{primitivo}--\eqref{primitivo1}. 
If $\mathbf H$ is an $r\times r$ invertible matrix, 
$$ \mathbf x_t =\pmb {\ \Lambda }^*\mathbf { F}^*_t +\pmb  \epsilon_t,\ \ \ \ (1-L) \mathbf { F}^*_t = \mathbf U^*(L) \mathbf u_t,$$
where $\mathbf F^*_t=\mathbf H^{-1} \mathbf F_t$,  $\pmb { \Lambda}^* = \pmb \Lambda\mathbf H $,  $\mathbf { U} ^*(L) = \mathbf H^{-1} \mathbf U(L)$,  
is another factor representation for the variables $x_{it}$.    The cointegrating vectors 
of  $\mathbf F^*_t$   are linear combinations of the vectors $\mathbf t^{h,k}$ and $\mathbf s^m$, so that primitive and trivial cointegration get mixed.  In particular, if the model is
estimated by principal components of the variables $x_{it}$, in general  the estimated factors $\mathbf {\hat F}_t$ approximate 
the space spanned by $\mathbf  F_t$, not $\mathbf F_t$ itself,  so that no distinction between trivial and primitive cointegrating vectors of $\mathbf {\hat F}_t$ is possible.

Secondly, as the elementary example below shows, not all vectors $\mathbf F_t$ can be put in the form \eqref{primitivo}--\eqref{primitivo1}. Let  $r=2$, $q=1$ and
$$ \mathbf U^*(L) = \begin{pmatrix} 1+L  \\ L^2 \\ \end{pmatrix}.$$
Suppose that there exists an invertible matrix 
$$ \mathbf H = \begin{pmatrix} \alpha & \beta \\ \gamma & \delta\\ \end{pmatrix},$$
such that  
$$ \mathbf U(L) = \mathbf H \mathbf U^*(L) = 
\begin{pmatrix} \alpha + \alpha L + \beta L^2\\ \gamma +\gamma L +\delta L^2\end{pmatrix}$$
has the form in \eqref{bellamattina}, third equation.  The second row of $\mathbf U(L)$ would equal the first multiplied by  $L$:
$$    \gamma +\gamma L +\delta L^2 =  (     \alpha + \alpha L + \beta L^2)L= \alpha L + \alpha L^2 +\beta L^3.$$ 
This implies that $\gamma=\beta=\alpha=\delta=0$. Thus no representation  \eqref{primitivo}--\eqref{primitivo1} exists for $\mathbf U^*(L)$.

\section{Representation theory for reduced rank I(1) vectors}\label{hauptbanhof}

\subsection{Families of cointegrated vectors}\label{crrv}

Consider the equation
\begin{equation}\label{wold1} (1-L) \pmb \zeta_t =\mbf S(L) ^{-1} \mbf C(L) \mbf u_t= \mbf U(L) \mbf u_t,\end{equation}
where $\det \mathbf S(L)$ has no roots of modulus less or equal to unity and $\mathbf C(1)\neq \mathbf 0$, so that $\mathbf U(1)\neq \mathbf 0$.
Suppose  that $u_{jt} \in L_2(\Omega,{\cal F},P)$,  for $j=1,\ldots,q$, where $(\Omega,{\cal F},P)$ is a probability space.     It is easily seen that all the solutions of (\ref{wold1}) are the processes 
\begin{equation}\label{volkswagen}\mbf F_t=\mbf {\tilde F}_t + \mbf W,\ \ \ t\in  \mathbb Z,\end{equation}
where $\mbf W$ is an $r$-dimensional  stochastic vector with   $W_{k}\in L_2(\Omega,{\cal F},P)$, $f=1,\ldots,r$,  and
\begin{equation}\label{fiat}\mbf{\tilde F}_t  =  \mathbf U(L) \pmb\nu_t=\mathbf S(L) ^{-1}\mathbf C(L) \pmb\nu_t, \ \ \ \  \hbox{where} \ \ \ \ 
\pmb \nu_t =\begin{cases} \hbox{$\mbf u_1+\mathbf u_2+ \cdots +  \mbf u_t$,  for  $t>0$}\\  \hbox{$\mathbf 0$, for $t=0$}\\
\hbox{$-( \mbf u_0+\mathbf u_{-1}+\cdots + \mbf u _{t+1})$, for $t<0$.}\\
\end{cases}
\end{equation}
Because $\mathbf { F} _t$ is a solution of \eqref{wold1},  $(1-L)\mathbf F_t=\mathbf U(L) \mathbf u_t$ is stationary with rational spectral density. Moreover,
as we assume $\mathbf C(1)\neq \mathbf 0$, so that   $\mbf U(1)=\mathbf S(1)^{-1}\mathbf C(1)\neq \mathbf 0$,    then $\mathbf F_t$  is $I(1)$. 

Now, because $\mbf S(1)^{-1}$ is   a non-singular $r\times r$ matrix, the cointegration rank of  $\mbf F_t$ only depends on the rank  of $\mbf C(1)$. Precisely, if $c$ is  the cointegration  rank of $\mbf F_t$, then   $c=r -{\rm rank}(\mathbf C(1))$,  so that
$r>c\geq r-q$.   Moreover, there exist  an $r\times (r-c)$ matrix   $\bm \xi $ and  a $q\times (r-c)$ matrix $\bm\eta$,  both of  full rank $r-c\leq q$, such that
\begin{equation}\label{casiniboia}\mbf C(1) =\bm \xi \bm \eta ',\end{equation}
 see \citet[p. 97, Proposition 3]{LT}.   The matrix  $\mbf C(L)$ has the  (finite) Taylor expansion
$$ \mbf C(L) =  \mbf  C(1)- (1-L)\mbf C'(1)+ \frac{1}{2}(1-L)^2\mbf C''(1)   - \cdots$$
Gathering all terms after the second and using (\ref{casiniboia}),
\begin{equation} \mbf C(L) = \bm \xi\bm \eta' - (1-L)\mbf C'(1)  + (1-L)^2 \mbf C_1(L),\label{monticioccolataio}\end{equation}
where $\mbf C_1(L)$ is a polynomial matrix.   

Representation (\ref{monticioccolataio}) can be used for a very convenient parameterization of $\mbf C(L)$.    

\begin{definition}\label{def2} \textbf{(Rational reduced-rank I(1) family with cointegration rank c)}
Assume that $r>q>0$, $r>c\geq r-q$ and let $\cal G$ be a set of  couples $(\mbf S(L),\mbf C(L))$, where: 

\begin{compactenum}[(i)]
	\item The matrix $\mbf C(L)$ has the  parameterization
\begin{equation}\label{montiperacottaro} 
\mbf C(L) = \bm \xi\bm \eta ' + (1-L) \mbf D+ (1-L)^2 \mbf E(L),\end{equation}
where $\bm \xi$ and $\bm \eta$ are $r\times (r-c)$ and $q\times (r-c)$ respectively,  $\mbf D$ is an $r\times q$ matrix and $\mbf E(L)$ is an $r\times q$ matrix polynomial of degree $s_1\geq 0$.
	\item $\mbf S(L) $ is an $r\times r$ polynomial matrix of degree $s_2\geq 0$. $\mbf S(0)=\mbf I_r$.
	\item  Denoting by  $\mathbf p$   the vector containing the  $\lambda =  (r-c) (r+q)+ rq(1+s_1)+r^2s_2$   coefficients of the matrices   in (\ref{montiperacottaro}) and in $\mbf S(L)$, we assume that $\mathbf p\in \Pi$, where  $\Pi$ is   an open subset of $\mathbb R^\lambda  $,   and that
for $\mathbf p\in \Pi$, if  $\det (\mbf S(z))=0$ then $|z|>1$.
\end{compactenum}
We say that the family of     processes $\mbf F_t$, such that  $(1-L)\mbf F_t =\mbf S(L)^{-1}\mbf C(L) \mbf u_t$,      where $\mbf u_t$ is a $q$-dimensional non-singular white noise and  $(\mbf S(L),\mbf C(L))$ belongs to  $\cal G$, is a \it rational reduced-rank $I(1)$ family with cointegration rank $c$.
\end{definition}

Three comments are in order.

\begin{enumerate}

\item[C1.]    Generically   the matrices $\pmb \xi$ and $\pmb \eta$  have full rank $r-c$, so that 
$\mathbf F_t$ is $I(1)$ with cointegration rank $c$.



\item[C2.]    In Remark \ref{bastapercarita}  in \ref{B},  we show that generically $\mathbf F_t$ has rank $q$.

\item[C3.]     Suppose that $r=4$, $q=1$, $s_1=s_2=0$.   For all $\mathbf p\in \Pi$, there exists a vector  $\mathbf d$  orthogonal to the $4$-dimensional columns 
$\pmb \xi\pmb\eta'$, $\mathbf D$, $\mathbf E_0$. Thus  $\mathbf d'\mathbf C(L)=\mathbf 0$ and therefore $\mathbf d(1-L) \mathbf F_t=0$, the degenerate case of cointegration 
mentioned in comment C9 on Definition \ref{iuno}.  However, if $s_2>0$ or $s_1$ is big enough as compared to $r$, degenerate contegration can be ruled out generically.

\end{enumerate}

Denoting by $\bm \xi_\perp$ an $r\times c$ matrix whose columns are linearly   independent  and orthogonal to all  columns of $\bm \xi$, the columns of $\bm \xi_\perp$  and $\pmb{\check \xi}_\perp=\mathbf S'(1)\pmb \xi_\perp $ are  full sets of independent cointegrating vectors for  $\mbf S(L)\mbf F_t$ and $\mathbf F_t$ respectively.

%
\subsection{Permanent and transitory shocks}\label{ctf}
Let $\mbf F_t$ be a rational reduced-rank $I(1)$ family with cointegration rank $c=r-q+d$, $q>d\geq 0$.
  Let $\bm \eta_\perp$ be a $q\times d$ matrix  whose columns are  independent and orthogonal to the columns of $\bm \eta$, and let 
$$\overline{\bm \eta}=\bm \eta( \bm \eta' \bm \eta)^{-1},\ \ \ \overline{\bm \eta}_\perp=\bm \eta_\perp(\bm \eta_\perp'\bm \eta_\perp)^{-1}.$$ 
Defining  $\mbf v_{1t}=  \bm \eta_\perp'  \mbf u_t$, and $\mbf v_{2t}=  \bm \eta'  \mbf u_t$,   we have 
$$\mathbf u_t=\overline{\bm \eta}_\perp \mathbf v_{1t} + \overline{\bm \eta}\mathbf v_{2t}= \begin{pmatrix} \overline{\bm \eta}_\perp & \overline{\bm \eta}\end{pmatrix}
\begin{pmatrix} \mathbf v_{1t}\\ \mathbf v_{2t}\\ \end{pmatrix}$$
We have 
\beq\label{sexyboschi}\mathbf C(L) \mathbf u_t=
\left [\mbf C(L)\left ( \overline{\bm \eta}_\perp \ \overline{\bm\eta}\right )\right ]  \begin{pmatrix} \mathbf v_{1t}\\ \mathbf v_{2t}\\ \end{pmatrix}=
(1-L) \mbf G_1(L) \mbf v_{1t}+ 
\left (  \bm \xi +(1-L) \mbf G_2(L) \right ) \mbf v_{2t}.\eeq
where 
$\mbf G_1(L) = \left ( \mbf D + (1-L) \mbf E(L)\right ) \overline{\bm \eta}_\perp$, and 
$\mbf G_2(L) = \left ( \mbf D + (1-L) \mbf E(L)\right ) \overline{\bm\eta}$. 
 Using \eqref{sexyboschi}, we can write all the solutions of the difference equation $(1-L) \mbf F_t=\mbf S(L) ^{-1} \mbf C(L) \mbf u_t$ as 
\beq\label{PTdecomposition}
\mbf F_t = \mbf  S(L)^{-1} \left[\mbf G_1(L)\mbf v_{1t} + 
 \mbf G_2(L)\mbf v_{2t} +  \mbf T_t \right]+\mbf W ,
 \eeq
 where $\mbf W\in  L_2(\Omega,{\cal F},P)$, and 
 $$\mbf T_t=\begin{cases}
 \hbox{$\bm{\xi}(\mbf  v_{21} +\mbf v_{22} +\cdots + \mbf v_{2t})$, for $t>0$}\\ \hbox{$\mathbf 0$, for $t=0$}\\
 \hbox{$- \bm {\xi} (\mbf v_{20}+\mbf v_{2,-1}+\cdots + \mbf v_{2, t+1})$, for $t<0$.}\\ \end{cases}
   $$
As $\bm\xi$ is full rank, we see that $\mbf F_t$ is driven by  the $q-d=r-c$ permanent shocks $\mbf v_{2t}$, and by the $d$ temporary shocks  $\mbf v_{1t}$. In representation \eqref{PTdecomposition}, the component  
$\mbf T_t$ is the common-trend of \citet{stockwatson88JASA}. Note that the number of permanent shocks is obtained as $r$ minus the cointegration rank, as usual. However, the number of transitory shocks is obtained as the complement of the number of  permanent shocks  to $q$, not to $r$, as though $r-q$ transitory shocks had a zero coefficient.

\subsection{Error Correction  representations}\label{ECM}
We now prove  the Granger Representation Theorem for  singular stochastic vectors
 $\mbf F_t$  belonging to   a rational reduced-rank $I(1)$ family with cointegration rank $c$ and parameters in the open set $\Pi\in  \mathbb R^\lambda$.   Our line of reasoning combines     arguments used in the  proof  of Granger's Theorem in the non-singular case (see e.g. 
  \cite{Johansen95}, Theorem 4.5, p. 55-57)  and the results in Proposition \ref{propp1}  on singular stochastic vectors.

   From 
  Definition \ref{def2}, $\mathbf F_t$ is a solution of the equation
$$ \mathbf S(L)   (1-L) \pmb \zeta_t =\mathbf C(L) \mathbf u_t= \left ( \pmb \xi \pmb \eta ' +(1-L) \mathbf D+(1-L)^2 \mathbf E(L) \right ) \mathbf u_t,$$
and has the representation $\mathbf F_t =\mathbf {\tilde F}_t+\mathbf W$, where 
$\mathbf {\tilde F}_t$ is defined in \eqref{fiat} and $\mathbf W$ is an $r$-dimensional stochastic vector.

Generically, the matrix  $\bm \zeta=\begin{pmatrix} \bm \xi_\perp'\\ \bm \xi'\end{pmatrix}$  is $r\times r$ and invertible (see comment C1 on Definition \ref{def2}). We have
\begin{equation}\label{compagnalanzillotta}\begin{aligned}
(1-L) \bm \zeta\mbf S(L) \mbf F_t&= \pmb\zeta \mathbf C(L) \mathbf u_t=
\left \{ \begin{pmatrix} \mbf 0 _{c\times q}\\ \bm \xi ' \bm \xi \bm \eta'\end{pmatrix} + (1-L)
\begin{pmatrix} \bm \xi_\perp ' \mbf D\\ \bm \xi ' \mbf D\\ 
  \end{pmatrix}+ (1-L)^2\begin{pmatrix} \bm \xi_\perp ' \mbf E(L)\\ \bm \xi ' \mbf  E(L)\\ 
  \end{pmatrix} \right \} \mbf u_t\\ &=
\begin{pmatrix} (1-L)\mbf I_c & \mbf 0\\ \mbf 0 &\mbf I_{r-c}\\ \end{pmatrix}   \left \{ \begin{pmatrix} \bm \xi_\perp ' \mbf D\\ \bm \xi ' \bm \xi \bm \eta'\end{pmatrix} + (1-L)
\begin{pmatrix}  \bm \xi_\perp' \mbf E(L) \\
\bm \xi ' \mbf D\\ 
  \end{pmatrix}+ (1-L)^2\begin{pmatrix}\mbf 0 _{c\times q}  \\ \bm \xi ' \mbf  E(L)\\ 
  \end{pmatrix} \right \} \mbf u_t  .
\end{aligned}\end{equation}
Taking the first $c$ rows,
\begin{equation}\label{oddiosanto} (1-L) \bm \xi '_\perp \mbf S(L)\mbf F_t = 
(1-L) \left ( \bm \xi_\perp ' \mbf D + (1-L) \bm \xi'_\perp \mbf E(L)\right )\mbf u_t.\end{equation}
In  \ref{A} we prove that if  $\mathbf F_t$ is such that 
\eqref{oddiosanto} holds and $ \bm \xi '_\perp \mbf S(L)\mbf F_t$ is weakly stationary with rational spectral density, then, in   $\mathbf F_t= \mathbf {\tilde F}_t+\mathbf W$, $\mathbf W$  must be chosen such that
%
\begin{equation}\label{oddiosanto1}\bm \xi '_\perp\mbf S(L) \mbf F_t  =\mbf k +\left ( \bm \xi_\perp ' \mbf D + (1-L) \bm \xi'_\perp \mbf E(L)\right )\mbf u_t,\end{equation}
where $\mbf k$ is a $c$-dimensional  constant vector and $\pmb \xi_\perp ' \mathbf S(1) \mathbf W=\mathbf k$ . 
Now,
\begin{equation}\label{oddiosanto4} \bm \xi_\perp ' \mbf S(1) \mbf F_t = 
\bm \xi_\perp ' \mbf S(L) \mbf F_t  - \bm \xi _\perp '  \mbf  S^*(L) (1-L) \mbf F_t,\end{equation}where $\mathbf S^*(L)$  is  the polynomial  $(\mathbf S(L) -   \mathbf S(1))/(1-L) $. Obviously  $\pmb \xi'_\perp \mathbf  S^*(L) (1-L) \mbf F_t$ is weakly stationary with rational spectral density.
As a consequence, $\bm \xi_\perp ' \mbf S(L) \mbf F_t $ is weakly stationary with rational spectral density if and only if $\bm \xi_\perp ' \mbf S(1) \mbf F_t $   
 is weakly stationary with rational spectral density.   Moreover,
   equation   \eqref{oddiosanto4},   replacing $\pmb\xi_\perp'\mathbf S(L) \mathbf F_t$ with the left-hand side of \eqref{oddiosanto1} and $(1-L) \mathbf F_t$ with 
 $ \mathbf S(L) ^{-1}(\pmb \xi\pmb \eta' + (1-L)\mathbf D+(1-L^2\mathbf E(L) ) \mathbf u_t$, becomes
 $$ \bm \xi_\perp ' \mbf S(1) \mbf F_t = \mathbf k+ \left \{ \left (\pmb \xi'_\perp \mathbf D -\pmb \xi'_\perp\mathbf S^*(1)\mathbf S(1)^{-1} \pmb \xi \pmb \eta'\right )+ (1-L) \pmb{\cal M} (L)\right \}\mathbf u_t.$$
As  $\pmb \xi'_\perp \mathbf D -\pmb \xi'_\perp\mathbf S^*(1)\mathbf S(1)^{-1} \pmb \xi \pmb \eta'\neq \mathbf 0$ generically,   $\bm \xi_\perp ' \mbf S(1) \mbf F_t$ is generically $I(0)$.

In conclusion,  assuming that $\mathbf F_t$ is such that $\pmb \xi'_\perp \mathbf S(L) \mathbf F_t$ is weakly stationary with rational spectral density and mean $\mathbf k$:
$$
\begin{aligned}\begin{pmatrix} \mbf I_c  & \mbf 0\\ \mbf 0 & (1-L) \mbf I_{r-c}\end{pmatrix} &\bm \zeta\mbf S(L) \mbf F_t = \\ &
\begin{pmatrix}\mbf k \\ \mbf 0 _{(r-c)\times 1} \end{pmatrix} + 
  \left \{ \begin{pmatrix} \bm \xi_\perp ' \mbf D\\ \bm \xi ' \bm \xi \bm \eta'\end{pmatrix} + (1-L)
\begin{pmatrix}  \bm \xi_\perp' \mbf E(L) \\
\bm \xi ' \mbf D\\ 
  \end{pmatrix}+ (1-L)^2\begin{pmatrix}\mbf 0 _{c\times q}  \\ \bm \xi ' \mbf  E(L)\\ 
  \end{pmatrix} \right \} \mbf u_t  .\end{aligned}$$
  Denote by $\mbf M(L)$ the matrix between curly brackets. 
  The following statement is proved in \ref{B}.

\begin{prop}\label{montidc}  Assume that  the family of I(1) processes $\mbf F_t$, such that $(1-L)\mbf F_t=\mbf S(L)^{-1}\mbf C(L) \mbf u_t$, is a rational reduced-rank I(1)  family with cointegration rank $c$ and parameter set $\Pi$.  Then, for generic values of the parameters in $\Pi$,   the $r\times q$ matrix $\mbf M(z)$   is zeroless. In particular, generically, 
\begin{equation}\label{juveassassina}{\rm rank}\left (\mbf M(1)\right ) ={\rm rank}\left (   \begin{pmatrix} \bm \xi'_\perp \mbf D \\ \bm \xi' \bm \xi \bm \eta' \end{pmatrix}\right )=q.\end{equation}\end{prop}
For $r=q$,  \eqref{juveassassina}  is equivalent to
the    condition that  $\bm \xi_\perp' \mbf C^* \bm \eta_\perp$ has full rank in \cite{Johansen95}, Theorem 5.4, p. 55  (Johansen's  matrix $\mathbf C^* $ is equal to our $\mathbf D$).  
To see this,  observe that if $r=q$ then $\begin{pmatrix} \pmb   \eta_\perp &\pmb \eta\end{pmatrix}$ is $r\times r$  and invertible, and that
$$\mathbf M(1)=\begin{pmatrix} \pmb   \eta_\perp &\pmb \eta\end{pmatrix}^{-1}\begin{pmatrix}\pmb \xi'_\perp \mathbf D\pmb \eta_\perp &\pmb \xi'_\perp \mathbf D\pmb \eta \\ \mathbf 0_{r-c\times c} &\bm \xi' \bm \xi \bm \eta'\pmb \eta \end{pmatrix}.$$  
As $\bm \xi' \bm \xi \bm \eta'\pmb \eta$ is non singular, the determinant of $\mathbf M(1)$  vanishes if and only if 
the determinant of $\pmb \xi'_\perp \mathbf D\pmb \eta_\perp$ vanishes.

A consequence of Proposition \ref{montidc} and Proposition \ref{propp1}(I) is that generically there exists a finite-degree $r\times r$ polynomial matrix 
$$\mbf N(L)= \mbf I_r + \mbf N_1 L + \cdots + \mbf N_p L^p  ,$$ 
for some $p$, such that: (i)~$\mbf N(L)\mbf M(L)= \mbf M(0)$, i.e. $\mbf N(L)$ is a left inverse of $\mbf M(L) $; (ii)~all the roots of $\det (\mbf N(L))$ lie outside the unit circle, so that $\mbf N(1)$ has full rank.
  
In conclusion, for generic values of the parameters in $\Pi$, 
$$\mbf{ A}(L) \mbf F_t = \mbf h+\mbf C(0) \mbf u_t,$$
where    
\begin{equation}\label{fastidious3}\begin{aligned}\mbf { A} (L) &= 
\mbf I_r+ \mbf{ A}_1L+ \cdots +\mbf{ A}_PL^{P}=
 \bm \zeta ^{-1} \mbf N(L) \begin{pmatrix} \mbf I_c  & \mbf 0\\ \mbf 0 & (1-L) \mbf I_{r-c}\end{pmatrix} \bm \zeta\mbf S(L)\\ &= \bm \zeta^{-1} \mbf N(L) \begin{pmatrix} \bm \xi_\perp'\\ (1-L)\bm \xi'\end{pmatrix}\mbf S(L),\end{aligned} \end{equation} 
with $P=p+1+s_2$, and
\begin{equation}\label{fastidious}\mbf h =\mbf A(1)\begin{pmatrix} \mbf k
 \\   \mbf 0_{(r-c)\times1} \end{pmatrix}.  \end{equation}
Defining
 \begin{equation} \label{fastidious1} \bm \alpha=\bm \zeta^{-1} \mbf N(1) \begin{pmatrix} \mbf I_c \\ \mbf 0_{(r-c)\times c}  \end{pmatrix},\ \ \ \   \bm \beta =\mbf S(1)'\bm \xi_{\perp}, \end{equation}
both $\bm \alpha $ and $\bm \beta$ have generically  rank $c$ (regarding $\bm \alpha$, remember that $\mbf N(1)$ has full rank) and
 $\mbf A(1)=\bm \alpha \bm \beta'$.  Lastly, define
\begin{equation}\label{fastidious2}\mbf A^*(L) =(1-L) ^{-1} (\mbf A(L) -\mbf A(1) L).\end{equation}
We have proved the following statement.
 
\begin{prop} \label{grtrr} \textbf{(Granger Representation Theorem for reduced-rank I(1) vectors)} 
Let  $\mathbf F_t$ be a rational reduced-rank I(1)  family with cointegration rank $c$ and parameter set $\Pi$, so that 
$$(1-L)\mbf F_t=\mbf S(L)^{-1}\mbf C(L) \mbf u_t.$$     For generic values of the parameters in $\Pi$, 
(i)~$\mathbf F_t$ can be determined such that   $\bm \beta' \mbf F_t=\bm \xi_\perp ' \mbf S(1) \mbf F_t$  is weakly stationary with rational spectral density and mean $\mbf k$,   (ii)~$\mbf F_t$ has the Error Correction representation 
 \begin{equation}\label{faziocoglione}
 \mbf A(L) \mbf F_t = \mbf A^* (L) (1-L) \mbf F_t + \bm \alpha \bm \beta' \mbf F_{t-1}  = \mbf h + \mbf C(0) \mbf u_t.\end{equation}
 where the $r\times r$  finite-degree polynomial matrices $\mbf A(L) $ and $\mbf A^*(L)$,  the  full-rank $r\times c$ matrices $\bm\alpha $ and $\bm \beta$,  the $r$-dimensional constant vector $\mbf h$, have been defined above in \eqref{fastidious3}, \eqref{fastidious2}, \eqref{fastidious1}, \eqref{fastidious}, respectively.     Generically, $\pmb \beta '\mathbf  F_t$ is $I(0)$.
\end{prop}

%


In Definition \ref{def2} we have not assumed that $\mathbf u_t$ is fundamental for $(1-L)\mathbf F_t$. However, 
\begin{equation}\label{label} \mathbf C(L) = \pmb \zeta ^{-1} \begin{pmatrix}(1-L) \mathbf I_c & \mathbf 0 \\ \mathbf 0 & \mathbf I_{r-c}\end{pmatrix} \mathbf  M(L).\end{equation}
Therefore, by Proposition \ref{montidc},  generically the matrix $\mathbf C(L) $ has  full rank  $q$ for $|z|<1$.   Thus, see \cite{rozanov}, pp. 43--7:

\begin{prop}\label{montidc230}  Assume that  the family of I(1) processes $\mbf F_t$, such that $(1-L)\mbf F_t=\mbf S(L)^{-1}\mbf C(L) \mbf u_t$, is a rational reduced-rank I(1)  family with cointegration rank $c$ and parameter set $\Pi$.  For generic values of the parameters  the  vector $\mathbf u_t$ is fundamental for the vector $(1-L)\mathbf F_t$.
\end{prop}
As recalled in Section \ref{sop2}, if $(1-L) \mathbf F_t = \mathbf S(L)^{-1}\mathbf C(L) \mathbf u_t$, then   $\mathbf u_t$ is fundamental  for $(1-L) \mathbf F_t$  if it belongs to the space spanned by $(1-L)\mathbf F_\tau$, $\tau\leq t$. This is not inconsistent with the fact that when $c>r-q$,   so that $\mathbf C(z)$ is not  full rank for $z=1$, 
there are not autoregressive representations for  $(1-L)\mathbf F_t$, either finite or infinite,   and therefore no 
representations  of the form
$$\mathbf u_t=(\bm {\mathcal A}_0+ \bm{\mathcal A}_1   L +\bm{ \mathcal A}_2 L^2 +\cdots )(1-L) \mathbf F_t=\lim_{n\to \infty}\sum_{k=0}^n {\mathcal A}_k (1-L) \mathbf F_{t-k} ,$$
where the matrices $\bm{\mathcal A}_j$ are $q\times r$.
Indeed, fundamentalness of $\mathbf u_t$ only implies that $\mathbf u_t$ can be obtained as the limit of linear combinations of  the vectors $(1-L) \mathbf F_\tau$, $\tau\leq 0$, i.e.
$$ \mathbf u_t = \lim _{n\to \infty} \sum _{k=0}^\infty \bm {\mathcal B} ^{(n)} _k (1-L)\mathbf F_{t-k},$$
where the coefficient matrices depend both on $k$ and $n$.\footnote{For example, if $x_t=(1-L)u_t$, where $u_t$ is a univariate white noise, then $u_t$ is fundamental for $x_t$ although  no autoregressive representation $x_t+a_1 x_{t-1}+ a_2 x_{t-2}+\cdots= u_t$
exists. See 
\cite{brockwelldavis}, p. 111, Problem 3.8.\label{DAVIS}}

Lastly, let us recall that neither $\mathbf F_t$ nor $\mathbf u_t$ are identified in the factor model 
$$ \mathbf x_t = \pmb \Lambda \mathbf F_t+\pmb \epsilon _t,\ \ \  (1-L)  \mathbf F_t=\mathbf S(L)^{-1}\mathbf C(L) \mathbf u_t.$$
 In particular, if $\mathbf F_t=\mathbf H \mathbf { F}^*_t$, where  $\mathbf H$   is  $r\times r$ and invertible, 
$$ \mathbf x_t =\pmb {\ \Lambda }^*\mathbf { F}^*_t +\pmb  \epsilon_t,\ \ \ \ (1-L) \mathbf { F}^*_t = \mathbf S^*(L)^{-1} \mathbf  { C}^*(L) \mathbf u_t,$$
where $\pmb { \Lambda}^* = \pmb \Lambda\mathbf H $,  $\mathbf { S}^*(L) = \mathbf H^{-1} \mathbf S(L) \mathbf H$,  $\mathbf { C} ^*(L) = \mathbf H^{-1} \mathbf C(L)$ (see also Section \ref{tristru}).     In particular, if $\mathbf H^{-1}=\pmb \zeta\mathbf S(1) $,   the first 
$c$ coordinates  of 
$\mathbf { F}^*_t$   are $I(0)$ and the remaining $r-c=q-d$ are $I(1)$.\footnote{More precisely, the first $c$ equations in  $(1-L) \mathbf F^*_t = \mathbf S^*(L) ^{-1} (L) \mathbf C^*(L) \mathbf u_t$ have an $I(0)$ solution (the argument goes as in   the discussion of equation \eqref{compagnalanzillotta}). 
This is consistent with Definition \ref{iuno}(ii), some of the coordinates of a $I(1)$   vector can be $I(0)$.}  Moreover,  the $c$ coordinates of the  error vector
${\pmb \beta^*}' \mathbf F^*_{t-1}$  in representation \eqref{faziocoglione}  for $\mathbf F_t^*$    are linear combinations of the $I(0)$ factors  alone.


\subsection{VECMs and unrestricted VARs in the levels}\label{giorgisega}
Several papers  have addressed the issue  whether  
and when an Error Correction model or an unrestricted  VAR in the levels should be used for estimation in the case of non-singular cointegrated  vectors: \cite{simsstockwatson} have shown that 
the parameters of a cointegrated VAR are consistently estimated using an unrestricted  VAR in the levels; on the other hand, \citet{phillips98} shows that if  the variables are cointegrated, the long-run features of the impulse-response functions are consistently estimated only if the unit roots are explicitly taken into account, that is within a  VECM specification.   
The simulation exercise described below provides some evidence in favour of the VECM specification in the singular case. 

We generate $\mathbf F_t$ using  a specification of  \eqref{faziocoglione} with $r=4$, $q=3$, $d=2$, so that $c=r-q+d=3$. The $4\times 4$ matrix $\mathbf A(L)$ is of degree $2$.  Moreover,   the upper $3\times 3$  submatrix  of $\mathbf C(0)$ 
is lower triangular (see \ref{sec_dgp} for details).  We estimate a VECM as in \citet{johansen88,johansen91} and assuming $c$, the degree of $\mathbf A(L)$  and the identification  restrictions known.
  We replicate the generation $1000$ times for $T= 100,\ 500,\ 1000, \ 5000.$
For each replication, we estimate a (misspecified) VAR in  differences, a VAR in the levels and 
 a VECM, assuming known $c$ and the degree of $\mathbf A(L)$ and $\mathbf A^*(L)$. The  Root Mean Square Error  between estimated and actual 
 impulse-response functions  is computed  for each replication using all $12$ responses and averaged over all replications.  
 The results are   are shown in Table \ref{tab:simulations}.   We see  that  the RMSE of both the VECM and the LVAR  decreases 
 as $T$ increases. However,
 for all values of $T$,  the RMSE of the VECM stabilizes as the lag increases, whereas 
 it deteriorates for the LVAR, in line with the claim that the long-run rsponse of the variables are better estimated 
 with the VECM. 
 
 For the sake of simplicity, in the simulation exercise  the  ``structural shocks'' are determined by imposing restrictions  on the response of $\mathbf F_t$ to the shocks.  Precisely, the upper $3\times 3$ submatrix of $\mathbf C(0)$ is lower triangular. 
 However,   as we recalled in the Introduction and Section \ref{tristru},   
 neither the  factors $\mathbf F_t$   nor their response to the shocks,  have a direct
 economic interpretation.  In empirical work with actual data,  identification of the structural shocks  $\mathbf v_t$, which result as a linear transformation of $\mathbf u_t$, 
 is usually  obtained by imposing  restrictions on the impulse-response functions of the variables $x_{it}$ with respect to $\mathbf v_t$.

 \begin{table}[h]\caption{Monte Carlo Simulations. VECM: $r=4,  {\it q}$=3$, c=3$.}\label{tab:simulations}
\centering \footnotesize
\begin{tabular*}{.8\textwidth}{@{}@{\extracolsep{\fill}}ccccc|ccccc@{}}\hline \hline
&&&&&&&&&	\\[-9pt]															
	&	lags	&	DVAR	&	LVAR	&	VECM	&		&	lags	&	DVAR	&	LVAR	&	VECM	\\	[2pt]	\hline
&&&&&&&&&	\\[-11pt]			
\multirow{5}{*}{\rotatebox{90}{$T=100$}}	&	0	&	0.06	&	0.05	&	0.05	&	\multirow{5}{*}{\rotatebox{90}{$T=500$}}	&	0	&	0.02	&	0.02	&	0.02	\\	[2pt]	
	&	4	&	0.26	&	0.18	&	0.17	&		&	4	&	0.23	&	0.07	&	0.07	\\	[2pt]	
	&	20	&	0.30	&	0.37	&	0.22	&		&	20	&	0.25	&	0.14	&	0.09	\\	[2pt]	
	&	40	&	0.30	&	0.45	&	0.22	&		&	40	&	0.25	&	0.21	&	0.09	\\	[2pt]	
	&	80	&	0.30	&	0.57	&	0.22	&		&	80	&	0.25	&	0.32	&	0.09	\\	[2pt]	\hline
&&&&&&&&	\\[-11pt]																				
\multirow{5}{*}{\rotatebox{90}{$T=1000$}}	&	0	&	0.02	&	0.02	&	0.02	&	\multirow{5}{*}{\rotatebox{90}{$T=5000$}}	&	0	&	0.01	&	0.01	&	0.01	\\	[2pt]	
	&	4	&	0.23	&	0.05	&	0.05	&		&	4	&	0.22	&	0.02	&	0.02	\\	[2pt]	
	&	20	&	0.25	&	0.09	&	0.07	&		&	20	&	0.25	&	0.03	&	0.03	\\	[2pt]	
	&	40	&	0.25	&	0.13	&	0.07	&		&	40	&	0.25	&	0.04	&	0.03	\\	[2pt]	
	&	80	&	0.25	&	0.22	&	0.07	&		&	80	&	0.25	&	0.06	&	0.03	\\	[2pt]	\hline
\end{tabular*}

\begin{tabular}{p{.8\textwidth}}
\scriptsize Root Mean Squared Errors at different lags, when estimating the impulse response functions  of the simulated variables  $\mbf F_t$ to the common shocks $\mbf u_t$. Estimation is carried out using three different autoregressive representations: a VAR for $(1-L) \mbf F_t$ (DVAR), a VAR for $ \mbf F_t$ (LVAR), and a VECM with $c=r-q+d$ error correction terms (VECM). Results are based on 1000 replications.   For the data generating process see \ref{sec_dgp}. The  RMSEs are obtained averaging over all replications and all $4\times 3$ impulse responses.
\end{tabular}
\end{table}

\section{Cointegration of the variables $x_{it}$}\label{cointegrationDFM}
The relationship between  cointegration of  the factors $\mathbf F_t$ and  cointegration of the variables $x_{it}$ is now considered.   Let us firstly observe that, regarding model (\ref{unouno}), neither the assumptions (1) through (5)  listed in Section \ref{sop}, nor the asymptotic conditions (see e.g.   \cite{FGLR09}) say much on 
 the matrix $\pmb \Lambda$ and the vector   $\pmb\epsilon_t$  for a given finite $n$.   In particular, the first $r$   eigenvalues of  the matrix $\pmb \Lambda\pmb \Lambda'$ must diverge as $n\to \infty$, but this has no implications on the rank  of the matrix $\pmb \Lambda$ corresponding to, say, $n=10$.
Moreover, as we  see in Proposition \ref{spread}(iii), if the idiosyncratic components are $I(0)$,  then all $p$-dimensional  subvectors of $\mathbf x_t$ are cointegrated for  $p>q-d$, which is at odds with what is observed in the macroeconomic  datasets analysed in the empirical  Dynamic Factor Model literature.  This motivates assuming
that  $\pmb \epsilon _t$ is $I(1)$. In that case, see Proposition \ref{spread}(i), cointegration of $\mathbf x_t$ requires that both the common and the idiosyncratic components are cointegrated.  Some results 
are collected in the statement below.

\begin{prop}\label{spread}  Let $\mathbf x^{(p)}_t= \pmb \chi^{(p)}_t+\pmb \epsilon^{(p)}_t=\pmb \Lambda ^{(p)}\mathbf F_t +\pmb \epsilon^{(p)}_t$ be a $p$-dimensional  subvector of $\mathbf x_t$, $p\leq n$. Denote by $c_\chi^p$ and $c_\epsilon ^p$ the cointegration rank of $\pmb \chi^{(p)}_t$ and $\pmb \epsilon^{(p)}_t$ respectively.  Both range from $p$, stationarity, to $0$, no cointegration.

\begin{enumerate}
\item[(i)] $\mathbf x^{(p)}_t$  is cointegrated   only if  $\pmb \chi^{(p)}_t$ and $\pmb \epsilon^{(p)}_t$ are both cointegrated.  

\item[(ii)] If $p>q-d$ then $\pmb \chi^{(p)} _t$ is cointegrated. If $p\leq q-d$ and $\mathbf \Lambda^{(p)} $ is full rank then 
$\pmb \chi^{(p)}_t$ is not cointegrated.  If $p\leq q-d$ and rank$(\mathbf \Lambda^{(p)})<p $ then $\pmb \chi^{(p)}_t$ is cointegrated.
 
\item  [(iii)]~Let 
$V^\chi\subseteq \mathbb R^p$ and $V^\epsilon\subseteq \mathbb R^p$ be the cointegration spaces of $\pmb \chi^{(p)}_t$ and $\pmb \epsilon^{(p)}_t$ respectively.  The vector $\mathbf x^{(p)}_t $ is cointegrated if and only if 
the intersection of $V^\chi$ and $V^\epsilon$   contains non-zero vectors. In particular,
if $p>q-d$ and $c^\epsilon>q-d$ then $\mathbf x^{(p)} $ is  cointegrated.


\end{enumerate}
\end{prop}

\begin{proof}
Because $\chi_{it}$ and $\epsilon_{js}$ are orthogonal  for all $i,j,t,s$, see Assumption (5) for model \eqref{unouno},  the spectral densities of $\mathbf (1-L)\mathbf x^{(p)}_t$, $(1-L)\pmb\chi^{(p)}_t$, $(1-L)\pmb\epsilon^{(p)}_t$ fulfill:
\begin{equation} \label{juveperderai}\pmb \Sigma_{\Delta x}^{(p)}(\theta) = \pmb\Sigma_{\Delta \chi}^{(p)}(\theta) +\pmb \Sigma_{\Delta \epsilon}^{(p)}(\theta)\ \ \  \theta\in [-\pi,\pi].\end{equation}
  Now, \eqref{juveperderai} implies that 
\begin{equation}\label{cicerone1}\lambda_p\left (\pmb \Sigma^{(p)}_{\Delta x}(0)\right )\geq \lambda_p\left (\pmb \Sigma^{(p)}_{\Delta \chi}(0)\right )+ \lambda^{(p)}\left (\pmb \Sigma^{(p)}_{\Delta\epsilon}(0)\right ),\end{equation}
where $\lambda_p(A) $  denotes the smallest eigenvalue of the hermitian matrix $A$; this is one of the Weyl's inequalities, see \cite{Franklin}, p. 157, Theorem 1.   Because  spectral density  matrices are  non-negative definite, the right hand side  in \eqref{cicerone1} vanishes 
if and only if both terms on the right hand side vanish,  i.e. the spectral density of $\Delta \mathbf x_t^{(p)}$ is singular at zero if and only if 
the spectral densities of $\Delta \pmb \chi_t^{(p)}$ and $\Delta\pmb \epsilon_t^{(p)}$ are singular at zero. By definition \ref{iuno}, 
  (i) is proved.
 
Without loss of generality we can assume that $\mbf S(L)=\mbf I_r$.  By  substituting \eqref{PTdecomposition} in \eqref{unouno}, we obtain
\begin{equation}\label{PTtemp}
\mbf x_t = \bm \Lambda\left [ \left(\mbf G_1(L)\mbf v_{1t} +  \mbf G_2(L)\mbf v_{2t} +  \mbf T_t \right )+\mbf W\right]+ \bm \epsilon_t,
\end{equation}
where on the right hand side the only non-stationary term is $\mbf T_t$ and (possibly) $\bm\epsilon_t$. By recalling that $\mbf T_t=\bm\xi \sum_{s=1}^t \mbf v_{2s}$ where $\bm\xi$ is of dimension $r\times (q-d)$ and rank $q-d$, and by defining $\mathbfcal{G}_t=\bm \Lambda[\mbf G_1(L)\mbf v_{1t} +  \mbf G_2(L)\mbf v_{2t}+\mbf Z]$  and ${\mathbfcal T}_t=\sum_{s=1}^t \mbf v_{2s}$, we can rewrite \eqref{PTtemp} as:
$$
\mbf x_t = \pmb \Lambda \pmb \xi{\mathbfcal T}_t + \mathbfcal{G}_t +\bm \epsilon_t .
$$
For $\mathbf x^{(p)}_t$:
$$ \mathbf x^{(p)} _t =\pmb\chi^{(p)}_t + \pmb \epsilon^{(p)}_t=\pmb  \Lambda ^{(p)} \pmb \xi{\mathbfcal T}_t +{\mathbfcal G}^{(p)}_t +\pmb \epsilon^{(p)}_t ,$$
where  $\pmb \Lambda^{(p)}$ and ${\mathbfcal G}^{(p)}_t $ have an obvious definition.  
Of course cointegration of the common components $\pmb\chi^{(p)}_t$ is  equivalent to cointegration of  $\pmb  \Lambda ^{(p)} \pmb \xi{\mathbfcal T}_t$, which in turn is equivalent to rank$(\pmb  \Lambda ^{(p)} \pmb \xi)<p$.    Statement (ii)  follows from 
$${\rm rank} \left (\pmb \Lambda ^{(p)} \pmb \xi\right  )\leq \min \left ({\rm rank} (\pmb \Lambda^{(p)}),{\rm rank}(\pmb \xi)\right ).$$

The first part of (iii) is obvious.  Assume now that $p>q-d$. If $c_\chi^p+c^p_\epsilon={\rm dim}(V^\chi)+{\rm dim}(V^\epsilon) = p-(q-d) + c^p_\epsilon >p$, i.e. if $c^p_\epsilon>q-d$, then the intersection between $V^\chi$ and $V^\epsilon$ is non-trivial, so that $\mathbf x^{(p)}_t$ is cointegrated.   \end{proof}

\section{Summary and conclusions}\label{conclusion} 
The paper studies representation theory for Dynamic Factor Models when the factors are $I(1)$ and singular. 
Singular $I(1)$ vectors are cointegrated, with cointegration rank $c$ equal to $r-q$, the dimension  of $\mathbf F_t$ minus its rank, plus $d,$ with $0\leq d<q$.
We prove that  if  $(1-L) \mathbf F_t$ has rational spectral density, then generically  $\mbf F_t$ has an Error Correction representation 
with $c$ error terms and a finite autoregressive matrix polynomial.  Moreover, $\mathbf F_t$ is driven by $r-c$ permanent shocks and 
$d$ transitory shocks, with $r-c+d=q$, not $r$ as in the  non-singular case.  These results are obtained by combining the standard 
results on cointegration with recent results on singular stochastic vectors.
 
Using simulated  data generated by a simple singular VECM,    confirms   previous results, obtained for non-singular vectors,  showing that under cointegration the long-run features of impulse-response  functions are better estimated using a VECM rather than a VAR in the levels.

In Section \ref{cointegrationDFM} we argue that stationarity of the idiosyncratic components would produce 
an amount of cointegration for the observable variables $x_{it}$ that is not observed in the datasets that are  standard
 in Dynamic Factor Model literature, see e.g.  \cite{stockwatson02JASA,stockwatson02JBES,stockwatson05},
 \citet{FGLR09}. Thus the idiosyncratic  vector  in those  datasets is likely to be $I(1)$ and, in light of this, an estimation strategy robust to the assumption that some of the idiosyncratic variables $\epsilon_{it}$ are $I(1)$ should be preferred.

The results in this paper are the basis for estimation of $I(1)$ Dynamic Factor Models with cointegrated factors, which is developed in the companion paper \citep{BLLirf}.

\bibliographystyle{chicago}

\bibliography{BLL_biblio}


\clearpage
\renewcommand\appendix{\par
\setcounter{section}{0}%
\setcounter{subsection}{0}
\setcounter{equation}{0}
\setcounter{table}{0}
\setcounter{figure}{0}
\gdef\thesection{Appendix \Alph{section}}
\gdef\thefigure{\Alph{section}\arabic{figure}}
\gdef\theequation{\Alph{section}\arabic{equation}}
\gdef\thetable{\Alph{section}\arabic{table}}
}
\appendix

\section{Proofs}\label{sec_proofs}

\subsection{Stationary solutions of $(1-L) \mbf  y_t = (1-L) \pmb \zeta_t$}\label{A}
Consider firstly the difference equation  in the unknown $g$-dimensional  vector  process $\pmb \zeta_t$:  
\begin{equation}\label{eq:appe} (1-L) \pmb\zeta_t =\mathbf w_t, \ \ \  t\in \mathbb Z,\end{equation}
where $\mathbf w_t$ is a $g$-dimensional stochastic  process  with  $w_{jt}\in L_2(\Omega,{ F},P)$.  
Define a solution 
 of \eqref{eq:appe} as a process $\mathbf y_t$, $t\in \mathbb Z$,   $Y_{jt}\in L_2(\Omega,{\cal F},P)$, such that $(1-L) \mathbf Y_t=\mathbf w_t$.  If 
 $\mathbf {\tilde y}_t$ is a solution, 
 then all the solutions of \eqref{eq:appe} are $\mathbf{\tilde y}_t+ \mathbf W$, where $\mathbf W$ is a $g$-dimensional  stochastic variable with  $W_{jt}\in L_2(\Omega,{\cal F},P)$, i.e. a particular solution plus  a constant stochastic process (a solution of the homogeneous equation  $(1-L) \mathbf y_t=\mathbf 0$).
  
Assume that $\mbf z_t$ is a $g$-dimensional, weakly stationary process with  a moving-average  representation $\mbf z_t = \bm{Z}(L) \mbf v_t$,   where:

\noindent  (i)~$\mbf v_t$ is an $s$-dimensional  non-singular white noise and $s\leq g$
belonging to   $L_2(\Omega ,{\cal F}, P)$, 

\noindent (ii)~$\bm{Z}(L)$ is a $g\times s$  square-summable matrix such that $ z_{kt}$  and $v_{jt}$, for $t\in \mathbb Z$, $k=1,2,\ldots,g$, $j=1,2,\ldots,s$, span the same subspace of  $L_2(\Omega ,{\cal F}, P)$.\footnote{A very weak condition. However, if $\mathbf Z(L)$ is a band-pass filter, (ii) does not hold.}   

Now consider the equation
\begin{equation}\label{lanzillotta} (1-L )\pmb \zeta_t = (1-L) \mbf z_t.\end{equation}
Because $\mbf z_t$ trivially fulfills (\ref{lanzillotta}),   all the solutions of \eqref{lanzillotta} are
  $$ \mbf y_t= \mbf z_t+\mathbf K$$
where $\mbf K$  is a $g$-dimensional stochastic  vector belonging to $L_2(\Omega ,{\cal F}, P)$.
We want to determine the conditions such that the solution $\mathbf y_t$ is weakly stationary with a spectral density.

 Let 
$$ \mbf K = \sum_{k=-\infty}^\infty \mathbf P_k\mathbf v_k+ \mbf H.$$
be the orthogonal projection of $\mbf K $ on the space spanned by $\mbf v_k$, $k\in \mathbb Z$.   Setting $\mathbf V= \sum_{k=-\infty}^\infty \mathbf P_k\mathbf v_k$, 
because $\mathbf H$ is orthogonal to $\mathbf V$ and $\mathbf z_\tau$, $\tau\in \mathbb Z$,
$$  {\rm E}(\mbf y_t \mbf y'_{t-k}) = {\rm E} (\mbf z_t \mbf z'_{t-k})+{\rm E} (\mbf V\mbf V') + {\rm E}( \mbf H \mbf H') +{\rm E}(\mbf z_t  \mbf V') +{\rm E} ( \mbf V \mbf z'_{t-k}) . $$
Given $k$, the last two terms tend to zero when $t$ tends to infinity (by the same argument used to prove that the autocovariances of a moving average tend to zero as the lag tends to infinity).  Weak stationary of $\mathbf y_t$   implies that
$${\rm E}(\mbf z_t  \mbf V') +{\rm E} ( \mbf V \mbf z'_{t-k})=0,$$
for all  $t\in \mathbb Z$.    On the other hand, given $t$,  ${\rm E} ( \mbf V \mbf z'_{t-k})$ tends to zero as $k$ tends to infinity (again the argument on autocovariances), so that 
${\rm E}(\mbf z_t  \mbf V')=0$ for all $t$.   Orthogonality of $\mathbf V$ to all $\mathbf z_t$ implies orthogonality to all $\mathbf v_t$, see assumption  (ii) above. 
 As $\mbf V$ is an average of  $\mbf v_k$, $k\in \mathbb Z$, then $\mbf V=\mathbf 0$.   In conclusion, all the stationary solutions of (\ref{lanzillotta})  are $\mbf y_t =\mbf z_t+\mathbf K$ with $\mbf K$ orthogonal to $\mbf z_t$ for all $t\in \mathbb Z$.

Lastly, the spectral measure  of $\mathbf y_t$ has a jump at frequency zero unless the variance-covariance matrix of $\mathbf K$ is zero.
Thus  $\mbf y_t$ has a spectral density if and only if $\mbf K$ is a constant vector (i.e.   $\mathbf K(\omega)=\mathbf k$ almost surely in $\Omega$).  In that case the spectral densities of $\mbf y_t$ and  $\mbf z_t$ coincide. 

Using the results above we can prove the statement in comment C7 on  Definition \ref{iuno}, Cointegration. 
If $\mathbf y_t$ is such that  $(1-L) \mathbf y_t = \mathbf V(L) \mathbf v_t$, then 
 $\mathbf y_t =\mathbf {\tilde y}_t + \mathbf W$, where
$$\mbf{\tilde y}_t  =  \mathbf V(L) \pmb\mu_t \ \ \ \  \hbox{where} \ \ \ \ 
\pmb \mu_t =\begin{cases} \hbox{$\mbf v_1+\mbf  v_2+ \cdots +  \mbf v_t$,  for  $t>0$}\\  \hbox{$\mathbf 0$, for $t=0$}\\
\hbox{$-( \mbf v_0+\mbf v_{-1}+\cdots + \mbf v_{t+1})$, for $t<0$}\\
\end{cases}
$$
and $\mathbf W$ is an $r$-dimensional stochastic vector. 
If
$\mathbf c$ is a cointegration vector,  then
$$ (1-L) \mathbf c' \mathbf  y_t = \mathbf c' \mathbf V(L) \mathbf v_t =  \mathbf c' \left [ \mathbf V(1) + (1-L) \frac{\mathbf V(L) -\mathbf V(1)}{1-L}\right ]\mathbf v_t= (1-L)\mathbf c'\mathbf V^* (L) \mathbf v_t,$$
where trivially the entries of $\mathbf V^*(L)=(\mathbf V(L) -\mathbf V(1)) / (1-L)$ are rational functions of $L$ with no poles of modulus less or equal to $1$.
From the last equation  we obtain 
$$\mathbf c'\mathbf y_t =\mathbf c' \mathbf {\tilde y}_t +\mathbf c'\mathbf W =  \mathbf c'\mathbf V^*(L) \mathbf v_t+w,$$
where  $w$ is a
stochastic variable.   The process $\mathbf  c'\mathbf y_t$ is weakly stationary with rational spectral density if and only if $w$ is a constant  with probability one.  On the other hand,
$$ \mathbf  c'\mathbf {\tilde y} _t = \mathbf c '\mathbf V(L) \pmb \mu_t = (1-L) \mathbf c' \mathbf V^* (L) \pmb \mu_t= \mathbf c' \mathbf V^* (L) \mathbf v_t,$$
so that $\mathbf c' \mathbf W= w$.  In conclusion,  $\mathbf c' \mathbf y_t$ is weakly stationary with rational spectral density if and only if, in the solution 
$\mathbf y_t= \mathbf {\tilde y}_t+ \mathbf W$, the stochastic vector
 $\mathbf W$ is chosen such that 
$\mathbf c' \mathbf W$ is constant with probability one.

The same reasoning applies to equation
 \eqref{oddiosanto}, to prove that 
 $\bm \xi '_\perp \mbf S(L)\mbf F_t $  is weakly stationary with  rational  spectral density  if and only if  \eqref{oddiosanto1} holds and 
 $\pmb \xi_\perp '\mathbf S(1) \mathbf W=\mathbf k$.


%
%


%
\subsection{Proof of Proposition \ref{montidc}}\label{B}

With one exception at the end of the proof, we keep using  the notation  $\mathbf C(z)$, $\mathbf M(z)$, etc., avoiding explicit dependence on $\mathbf p\in \Pi$ (see Definition \ref{def1}).

\begin{rem}\label{rem1}  Suppose that the statement $S(\mathbf p)$, depending on a  vector $\mathbf p\in \Pi$, is equivalent to a set of polynomial equations for the parameters, for example the statement that all the $q\times q$ minors of $\mathbf M(1)$  vanish, i.e. that  ${\rm rank} (\mathbf M(1)) <q$. Statement $S(\mathbf p)$ is true either for a nowhere dense subset of $\Pi$ or for the whole $\Pi$. Thus, if  the statement is false for one point in $\Pi$, it is generically false in $\Pi$. Moreover, $S(\mathbf p)$ can be obviously extended to any $\mathbf p\in \mathbb R^\lambda$ and, as $\Pi$ is an open subset of $\mathbb R^\lambda$, if the statement $S$ is false for one point in $\mathbb R^\lambda$,   then it is  generically false in  $\mathbb R^\lambda$ and therefore in $\Pi$.
\end{rem}

\begin{rem}\label{bastapercarita} Remark \ref{rem1} can be used to show that generically the stochastic vectors of a reduced rank $I(1)$ family, see 
Definition \ref{def2},  are of rank $q$.   Let $\mathbf p^*$ be  a point in $\mathbb R^\lambda$ such that $\pmb \xi=\mathbf 0$, $\pmb \eta=\mathbf 0$, $\mathbf S(L)=\mathbf I_r$, $\mathbf E(L)=\mathbf 0$ and let $\mathbf D$ be of rank $q$ (note that $\mathbf p^*$ does not necessarily belong to $\Pi$).  Given $\theta^*\in [-\pi,\ \pi]$, $\theta^*\neq 0$,  the matrix $\mathbf S(e^{-i\theta^*})^{-1}\mathbf C(e^{-i\theta^*})$ is  equal to $(1-e^{-i\theta^*})\mathbf D$ at $\mathbf p^*$ and  has therefore  rank
$q$, so that, by Remark \ref{rem1}, has rank $q$ generically in $\Pi$. Thus, generically in $\Pi$, the spectral density of $(1-L) \mathbf F_t$, which is a rational function of $e^{-i\theta}$,  has   rank $q$ except for  a finite subset of $ [-\pi,\ \pi]$ (depending  on $\mathbf p$).

\end{rem}

\begin{rem}\label{rem2}
Consider the polynomials
$$ A(z) = a_0 z^n+a_1z^{n-1}+\cdots+ a_n,\ \ \  B(z) =b_0 z^m+b_1 z^{m-1}+ \cdots + a_m$$  
and let  $\alpha_i$, $i=1,\ldots,n$ and $\beta_j$, $j=1,\ldots,m$, be the roots of $A$ and $B$ respectively.
Suppose that $a_0\neq 0$ and $b_0\neq 0$. 
Then, see \citet[pp. 83-8]{vanderWaerden},  
$$a_0^mb_0^n \prod_{i,j}(\alpha_i-\beta_j)=R(a_0,a_1,\ldots,a_n; b_0,b_1,\ldots,b_m),$$
where $R$ is a polynomial function.  The function $R$ is called the resultant 
of $A$ and $B$. 
The resultant vanishes if and only if  $A$ and $B$ have a common  root. Now suppose that the coefficients $a_i$ and $b_j$ are polynomial functions of $\mathbf p\in \Pi$.  Then, by Remark \ref{rem1},  if there exists  a point $\mathbf {\tilde p}\in \Pi$  (or $\mathbf{ \tilde p}\in \mathbb R^\lambda$) such that $a_0(\mathbf {\tilde p})\neq 0$, $b_0(\mathbf{\tilde p})\neq 0$, and $R(\mathbf{\tilde p})\neq 0$, then 
generically $A$ and $B$ have no common roots.
\end{rem}

\begin{rem}\label{rem3} 
Recall that  a zero of $\mbf M(z)$ is a complex number $z^*$  such that ${\rm rank}(\mathbf M(z^*)<q$ (see Proposition \ref{propp1}).  
 If 
$\mathbf M(z)$  has two $q\times q$ submatrices  whose determinants have no common roots, then 
$\mathbf M(z) $ is zeroless.
\end{rem}

 \medskip
Starting with 
 $$ \mbf C(z) = \bm \xi \bm \eta' + (1-z) \mbf D + (1-z)^2 \mbf E(z),$$
 we obtain, see Section \ref{ECM},
 $$\begin{aligned}\bm \zeta \mbf C(z)&=\begin{pmatrix} (1-z)\mbf I_c & \mbf 0\\ \mbf 0 &\mbf I_{r-c}\\ \end{pmatrix}   \left \{ \begin{pmatrix} \bm \xi_\perp ' \mbf D\\ \bm \xi ' \bm \xi \bm \eta'\end{pmatrix} + (1-z)
\begin{pmatrix}  \bm \xi_\perp' \mbf E(z) \\
\bm \xi ' \mbf D\\ 
  \end{pmatrix}+ (1-z)^2\begin{pmatrix}\mbf 0 _{c\times q}  \\ \bm \xi ' \mbf  E(z)\\ 
  \end{pmatrix} \right \}\\
  & =  \begin{pmatrix} (1-z)\mbf I_c & \mbf 0\\ \mbf 0 &\mbf I_{r-c}\\ \end{pmatrix}\mbf M(z).\\ \end{aligned}$$

With no loss of generality we can assume that $r=q+1$, see Remark \ref{rem3}. We denote  by $\mbf M_1(z)$ and $\mbf M_2(z)$ the  $q\times q$ matrices 
obtained by dropping the first  and the last row of $\mbf M(z)$ respectively. The degrees of  the polynomials $\det(\mbf M_1(z))$ and $\det(\mbf M_2(z))$ are $d_1= (q-d)(s_1+2)+d(s_1+1)$ and $d_2=(q-d-1)(s_1+2) + (d+1) (s_1+1)$.  

Let us now define a subfamily of  $\mbf M(z)$, denoted by $\mbf  {\underline M}(z)$, obtained by specifying $\pmb\eta$, $\pmb \xi$, $\pmb \xi'_\perp$, $\mathbf D$ and $\mathbf E(L)$ in the following way:
$$\bm {\underline \eta}'= \begin{pmatrix} \mbf 0_{(q-d)\times d}  \mbf I_{q-d}\\ \end{pmatrix},    \ \bm {\underline \xi} = \begin{pmatrix} \mbf I_{q-d} \\ \mbf 0 _{c\times (q-d)}
  \end{pmatrix},\  \ \bm {\underline \xi} _{\perp}' = \begin{pmatrix} \mbf K \\ \mbf H \end{pmatrix} , \ \mbf {\underline D} =\begin{pmatrix} \mbf H' & \mbf 0 _{(q+1)\times (q-d)}\end{pmatrix},\  \mathbf {\underline E} (z)= \begin{pmatrix}\mathbf E_1(z)\\ \mathbf E_2(z)\\ \mathbf E_3(z)\\ \end{pmatrix}, $$
where
$$
  \mbf K = \begin{pmatrix} \mbf 0 _{1\times(q-d)} &  1 & \mbf  0 _{1\times d}\end{pmatrix},\ \ \ \mbf H = \begin{pmatrix} \mbf 0_{d\times (q+1-d)} & \mbf I_d\end{pmatrix} ,$$
$$ \begin{aligned}\mbf E_{1}(z) &= \begin{pmatrix}\\& k_1(z)  & h_1(z) &\cdots &0 \\ \mbf 0_{(q-d)\times d}&  & \hskip-25pt\ddots&\hskip-25pt \ddots\\ &&&\hskip-29pt\ddots &  h_{q-d-1}(z)\\& 0 &\cdots & &
k_{q-d}(z) & &  \end{pmatrix} \\ \\ \mbf E_{2}(z) &= \begin{pmatrix}  e(z)&\mbf 0_{1\times (q-1)}  \end{pmatrix}\\  \\
\mbf E_{3}(z) &= \begin{pmatrix}f_1(z)  & g_1(z) &\cdots &0 \\  & \ddots&  \ddots&&\mbf 0_{d\times (q-d-1)}\\ 0 &\cdots & 
f_{d}(z) & g_{d}(z) &&  \end{pmatrix},
\\  
\end{aligned}$$  
the polynomial  entries   $e$, $k_i$, $h_i$, $f_i$ and $g_i$ being of degree  $s_1$.  We have:
$$\mbf{\underline M}(z) = \begin{pmatrix}\mbf 0_{1\times d } &  \mbf 0_{1\times (q-d)} \\ \mbf I_d & \mbf 0_{d\times(q-d)} \\ \mbf 0_{(q-d) \times d} & \mbf I_{q-d} \end{pmatrix} + (1-z) \begin{pmatrix} \mbf E_{2}(z) \\ \mbf E_{3}(z) \\ \mbf 0_{(q-d)\times q}\end{pmatrix}+ (1-z)^2 \begin{pmatrix} \mbf 0_{1\times q} \\ \mbf 0_{d\times q} \\ \mbf E_{1}(z) \end{pmatrix}.
$$
Notice that $\mathbf {\underline M}(z) $ has zero entries except for the  diagonal joining the positions $(1,1)$ and $(q,q)$, and the diagonal joining $(2,1)$ and $(q+1,q)$.  The matrices $\mbf{\underline M}_{1}(z)$ and $\mbf {\underline M}_{ 2}(z)$, obtained by dropping the first and the last row of $\mathbf {\underline M}(z)$, respectively,  are upper- and lower-triangular, respectively. Moreover,
$$\begin{aligned} \det (\mbf {\underline M}_{1}(z))&= [1+ (1-z) f_1(z) ]\cdots [(1  +(1-z)f_d(z) ] \\ &\ \ \ \ \ \ \ \times[1+(1-z)^2 k_1(z)]\cdots [1+(1-z)^2 k_{q-d}(z)]\\
\det (\mbf{\underline  M}_{ 2}(z))&=(1-z)^ {2q-d-1}e(z) [g_1(z) \cdots g_d(z)][  h_1(z)\cdots  h_{q-d-1}(z)]\\
\end{aligned}
$$
Now:
\begin{compactenum}[(i)]
	\item The leading coefficient of $\det (\mbf {\underline M}_{1}(z))$, call it $\underline Q_{1}$, corresponding to $z^{d_1}$, is the product of the leading coefficients of the polynomials $k_j(z)$, $j=1,\ldots, (q-d)$ and $f_i(z), $ $i=1,\ldots,d$.  Trivially,  there exist values for  the parameters of  the polynomials $k_j$ and $f_i$, such that $\underline Q_{1}\neq 0$.   Let $\pmb { \omega}_1$ be the vector of such parameters and $\mathbf {\underline M}^{\omega_1}_1(z)$ the matrix $\mathbf {\underline M}_1(z)$ corresponding to the parameters   in $\pmb \omega_1$.	
	\item Now observe, firstly,  that the polynomials $\det (\mbf {\underline M}_{1}(z))$ and $\det (\mbf {\underline M}_{2}(z))$  have no parameters in common,   and, secondly,  that the parameters of $\det (\mbf {\underline M}_{2}(z))$  vary in an open set (each one is  a subvector of the parameters vector of $\mathbf M(z)$, which varies in the open set $\Pi$).  As a consequence, there exist parameters for the polynomials $e$, $g_j$ and $h_i$,   call $\pmb \omega_2$ the vector of such parameters, such that   (1)~the leading coefficient  of  $\det (\mbf {\underline M}^{\omega_2}_{2}(z))$  does not vanish,  (2)~$\det (\mbf {\underline M}^{\omega_1}_{1}(z))$ and $\det (\mbf {\underline M}^{\omega_2}_{2}(z))$ have no roots in common.   This implies that, as the leading coefficient of $\det (\mbf {\underline M}^{\omega_1}_{1}(z))$ does not vanish as well, the resultant of $\det (\mbf {\underline M}^{\omega_1}_{1}(z))$  and $\det (\mbf {\underline M}^{\omega_2}_{2}(z))$  does not vanish, see Remark 2.
	\item Combining the parameters in $\pmb \omega_1$ and $\pmb \omega_2$, we determine a point $\mathbf{\underline p}\in \Pi$ such that, at $\mathbf {\underline p}$, the leading coefficient of $\det (\mbf {\underline M}_{1}(z))$ and  $\det (\mbf {\underline M}_{2 }(z))$
	and their resultant do not vanish.  As the leading coefficients  and the resultant 
	of $\det (\mbf { M}_{1}(z))$ and  $\det (\mbf {M}_{2 }(z))$ 
	are polynomial functions of the parameters in $\mathbf p$,  then
	by Remarks  2  and 3, $ \mbf M(z)$ is  generically  zeroless.
\end{compactenum}

\section{Non uniqueness}\label{nu} 
In Proposition \ref{grtrr} we prove that a singular $I(1)$ vector has a finite Error Correction representation with $c$ error correction terms. However, as anticipated in Section \ref{sop},  this representation is not unique since:
\begin{inparaenum}[(i)]
	\item different Error Correction representations can be obtained in which the number of error terms  varies between $d$ and $r-q+d$, 
	\item the left inverse of the matrix $\mbf M(L)$ may be not unique.
\end{inparaenum}
We discuss these two causes  of non-uniqueness for representation \eqref{faziocoglione} below.    In  \ref{unique1} we show that 
all such representation produce the same impulse-response functions.

\subsection{Alternative representations with different numbers of error correction terms}\label{unique2}
Let, for simplicity, $\mbf S(L)=\mbf I_r$ and consider the following example, with $r=3$, $q=2$, $c=2$, so that $d=1$:

$$\begin{aligned} \bm \xi' &=\begin{pmatrix} 1 & 1& 1\end{pmatrix} \\     \bm \eta'& = \begin{pmatrix} 1 & 2\\ \end{pmatrix} \\
                        \bm \xi_\perp ' &= \begin{pmatrix} 1 & -1 & 0\\ 0 & 1& -1\\ \end{pmatrix} \\                            
                               \end{aligned} $$  We have,
$$\begin{aligned}(1-L)\begin{pmatrix} \bm \xi_\perp'\\ \bm \xi'\end{pmatrix} \mbf F_t &= \begin{pmatrix} 1-L  & 0 & 0\\ 0& 1 - L& 0\\ 0&0&1\\ \end{pmatrix}\left \{ \begin{pmatrix} d_{11} - d_{21} & d_{12}- d_{22}\\     d_{21}-d_{31} & d_{22}-d_{32}\\  3 & 6\\ \end{pmatrix}     + (1-L) 
 \vphantom{\begin{pmatrix}1\\1\\1\\ \end{pmatrix}}  \mbf  G(L)
  \right \}   \mbf u_t,\end{aligned}$$      
where $(1-L) \mbf G(L)$ gathers the second and third terms  within curly brackets in the second line of (\ref{compagnalanzillotta}). If the first  matrix  within the curly brackets has full rank, we can proceed as in Proposition \ref{grtrr} and obtain an Error Correction representation with error terms
$$ \bm \xi _\perp' \mbf F_t= \begin{pmatrix} F_{1t} -F_{2t}\\ F_{2t} - F_{3t} \end{pmatrix}. $$   
However, we also have
$$\begin{aligned}(1-L)\begin{pmatrix} \bm \xi_\perp'\\ \bm \xi'\end{pmatrix} \mbf F_t &= \begin{pmatrix} 1-L  & 0 & 0\\ 0& 1 & 0\\ 0&0&1\\ \end{pmatrix}\left \{ \begin{pmatrix} d_{11} - d_{21} & d_{12}- d_{22}\\    (1-L)( d_{21}-d_{31}) & (1-L)(d_{22}-d_{32})\\  3 & 6\\ \end{pmatrix} \right. \\  &\left. + (1-L)  \mbf {\tilde G} (L) \right \}   \mbf u_t
   =\begin{pmatrix} 1-L  & 0 & 0\\ 0& 1 & 0\\ 0&0&1\\ \end{pmatrix} \mbf{\tilde M}(L) \mbf u_t.
  \end{aligned}$$     
 Assuming that the matrix
 $$\begin{pmatrix} d_{11} - d_{21} & d_{12 } - d_{22}  \\ 3 & 6\\ \end{pmatrix} $$
 is non-singular, the matrix $\mbf{\tilde M}(L)$ is zeroless and has therefore a finite-degree left inverse. Proceeding as in Proposition \ref{grtrr}, we obtain an alternative Error Correction representation with just one error term, namely $F_{1t}-F_{2t}$.     

This example can be generalized to show that generically $\mbf F_t$ admits Error Correction representations with a minimum $d$ and a maximum $r-q+d$ of error terms. In particular, if $d=0$, in addition to an Error Correction representation, $\mbf F_t$ generically has a finite-degree autoregressive representation with no error terms (i.e. a VAR), consistently with the fact that in this case $\mbf C(L)$ is generically zeroless.

 Experiments with simulated and  actual data suggest that the best results in estimation of singular VECMs  are obtained
 using $c$ (the maximum number of) error correction terms.

\subsection{The  left inverse of $\mathbf M(L)$  is not  necessarily unique}\label{unique3}
In the proof of Proposition \ref{grtrr} we have used the fact that generically the matrix $\mathbf M(L)$ has a finite-degree left inverse $\mathbf N(L)$. 
We now give some examples in which $\mathbf N(L)$ is not unique.
This is a   well known fact,   see also \cite{fornilippi10}, \cite{FHLZ2014}. 

Consider  
\begin{equation}\label{esempiuccio}(1-L) F_t = \begin{pmatrix} 1+aL\\1+bL\\ \end{pmatrix} u_t,\end{equation}
with $r=2$, $q=1$, $d=0$, $c=1$, with $a\neq b$.  In this case 
$\mbf A(L)$ is zeroless. An autoregressive representation 
can be obtained by elementary manipulations. Rewrite (\ref{esempiuccio}) as 
\begin{equation}\label{esempiuccio1} \begin{aligned} (1-L) F_{1t} &= u_t +a u_{t-1} \\
             (1-L)F_{2t}&= u_t+b u_{t-1}\\
             \end{aligned} 
\end{equation}
Taking $(b\ \  -a) \mbf C(L) \mbf u_t$,  we get
$$ u_t = \frac{\displaystyle b (1-L) F_{1t} -a(1-L) F_{2t}}{b-a}.$$
This can be used to get rid of $u_{t-1}$ in (\ref{esempiuccio1}) and obtain
\begin{equation}\left [ \mbf I_2 - \begin{pmatrix} \frac{\displaystyle ab}{\displaystyle b-a} & \frac{\displaystyle a^2 }{\displaystyle b-a} \\
 \frac{\displaystyle b^2}{\displaystyle b-a} & \frac{\displaystyle -ab }{\displaystyle b-a} \\   \end{pmatrix} L\right ](1-L) \mbf F_t=\begin{pmatrix}1 \\ 1\end{pmatrix} u_t,\label{projection}\end{equation}
which is an autoregressive representation in first differences. 

Model (\ref{esempiuccio1}), slightly modified, can be  used to illustrate non-uniqueness in the left inversion of $\mbf M(L)$. Consider
\begin{equation}\label{esempiuccio2} 
\begin{aligned} (1-L) F_{1t} &= u_t +a u_{t-1} \\
             (1-L)F_{2t}&= u_t+b u_{t-1}\\
             (1-L)F_{3t}&= u_t + c u_{t-1}.\\
\end{aligned} 
\end{equation}
Taking any vector $\mbf h =(h_1\ h_2\ h_3)$, orthogonal to $(a\ b\ c)$, we get rid of $u_{t-1}$ in (\ref{esempiuccio2}) and obtain an autoregressive representation in the differences. However, unlike in (\ref{esempiuccio1}), here the vectors $\mbf h $ span a $2$-dimensional space, thus producing an infinite set of autoregressive representations.

In the example just above non-uniqueness can also be seen as the consequence of the fact that the three stochastic variables $F_{j,t-1}$, $j=1,2,3$, are linearly dependent.  Therefore, projecting  each of the $F_{jt}$ onto the  space spanned by $F_{j,t-1}$, $j=1,2,3$, one would find a non-invertible covariance matrix, thus a unique projection of course but many representations of it as linear combinations of $F_{j,t-1}$, $j=1,2,3$.

 We do not address this problem systematically in the present paper. However, in the empirical analysis of \citet{BLLirf} we find no hint of singular covariance matrices.
\subsection{Uniqueness of impulse-response functions}\label{unique1}
Start with representation \eqref{wold1}
\begin{equation}\label{merkelmerkel}(1-L) \mathbf F_t= \mathbf S(L)^{-1} \mathbf C(L) \mathbf u_t=\mathbf U(L) \mathbf u_t=\mathbf  U_0\mathbf u_t +\mathbf U_1 \mathbf u_{t-1} +\cdots
=\mathbf  C(0)\mathbf u_t +\mathbf U_1 \mathbf u_{t-1} +\cdots
\end{equation} 
We assume that  $\mathbf u_t$ is fundamental for $(1-L) \mathbf F_t$, see Proposition \ref{montidc230}.
The impulse response function of $\mathbf F_t$ to 
$\mathbf u_t$ is 
$$ \mathbf H_j =\mathbf U_0+\cdots +\mathbf U_j,\ \ \ \ i=0,1,\ldots. $$
Now  suppose that $\mathbf F_t$ fulfills the autoregressive equation
\begin{equation}\label{laqualunque}\mathbf B(L) \mathbf F_t = (\mathbf I_r +\mathbf B_1 L+\ldots + \mathbf B_m L^m) \mathbf F_t=  \mathbf{\tilde  m} +\mathbf {\tilde R} \mathbf{\tilde  u}_t\end{equation}
where: (i)~$\mathbf{\tilde  R}$ is a full-rank  $r\times q$ matrix, (ii)~$\mathbf {\tilde u}_t$ is $q$-dimensional white noise, (iii)~$\mathbf {\tilde  u}_t$ is orthogonal 
to $(1-L) \mathbf F_\tau$, for $\tau\geq 0$. Applying $(1-L)$ to  both sides of (\ref{laqualunque}) we obtain
\begin{equation} \label{DAVIS21}\mathbf  B(L) \mathbf U(L) \mathbf u_t = (1-L) \mathbf {\tilde R} \mathbf {\tilde u}_t.\end{equation}

Assumption (i) and the argument mentioned in footnote \ref{DAVIS} imply that  $\mathbf{\tilde u}_t$ belongs to the space spanned by $\mathbf u_\tau$,  for $\tau\geq 0$, call it 
${\cal H}_{u,t}.$ Now consider the projection
$$ \mathbf {\tilde u}_t= \mathbf  G_0 \mathbf u_t +\mathbf G_1\mathbf u_{t-1}+\cdots $$
Multiplying both sides by $\mathbf u_{t-k}'$ and taking expected values:
$$ {\rm E} \mathbf {\tilde u}_t \mathbf  u_{t-k}'= \mathbf G_k \pmb \Gamma_u.$$
By assumption (iii),  $\mathbf {\tilde u}_t$ is orthogonal to ${\cal H}_{(1-L)F,t-k}$, for $k>0$,   which is equal  to ${\cal H}_{u,t-k}$, for $k>0$ (a 
consequence of the fundamentalness of $\mathbf u_t$ in \eqref{merkelmerkel}). Thus $\mathbf G_k=\mathbf 0$,  for $k>0$ and 
$$ \mathbf {\tilde u}_t = \mathbf G_0 \mathbf u_t,$$
where $\mathbf G_0$ is a non-singular $q\times q$ matrix.  Therefore, with no loss of generality, equation \eqref{laqualunque}  can be rewritten with 
$\mathbf R \mathbf u_t$ instead of $\mathbf {\tilde R}\mathbf{\tilde u}_t$ and \eqref{DAVIS21} becomes
$$\mathbf  B(L) \mathbf U(L) \mathbf u_t = (1-L) \mathbf { R} \mathbf { u}_t.$$
As $\mathbf u_t$ is  a non-singular $q$ dimensional white noise, this implies 
$$   \mathbf B(L) \mathbf U(L) = (1-L) \mathbf R,$$ 
so that:
$$\begin{aligned}& \mathbf U_0 =\mathbf R,\\ 
                       &\mathbf B_1 \mathbf R + \mathbf U_1 = -\mathbf R , \ \ \ \   \mathbf U_1 = -( \mathbf I_r+ \mathbf B_1) \mathbf R,\\
                      & \mathbf B_2 \mathbf R +\mathbf B_1\mathbf U_1 + \mathbf U_2 = \mathbf 0,\ \ \ \ \mathbf U_2 = (\mathbf B_1 + \mathbf B_1^2 -\mathbf B_2 )
                      \mathbf R,\\ &  \phantom{>}\vdots
                       \end{aligned}$$
and therefore
$$ \begin{aligned} &\mathbf H_0= \mathbf U_0 = \mathbf R,\\ &\mathbf H_1=\mathbf U_0 + \mathbf U_1 = -\mathbf B_1 \mathbf R,\\ &\mathbf H_2= \mathbf U_0+\mathbf U_1+\mathbf U_2 = \mathbf (\mathbf B_1^2 -\mathbf B_2)\mathbf R,\\ &\phantom{>}\vdots \end{aligned} $$
On the other hand, the impulse-response function implicit in \eqref{laqualunque} is given by the coefficient matrices of  $\mathbf K(L) \mathbf R$, where
$$\mathbf K(L) \mathbf B(L) = (\mathbf I_r + \mathbf K_1 L + \cdots ) \mathbf B(L) = \mathbf I_r.$$
It is easily seen that $\mathbf K_j \mathbf R=\mathbf H_j$.

Note that    we are not making assumptions  on $\mathbf B(1)$ in equation \eqref{laqualunque}.  When $d=0$,   
equation \eqref{laqualunque} can be  the autoregressive model in differences that results from left-inverting 
$\mathbf U(L)$ (no error correction term):
$$ \mathbf  {\tilde B}(L)
\mathbf   (1-L) \mathbf F_t =\mathbf B(L) \mathbf F_t= \mathbf C(0) \mathbf u_t.$$

  Replacing $\mbf u_t$ with any other white noise vector $\mbf w_t=\mbf Q \mbf u_t$, as we do when the shocks are identified according to  restrictions based  on economic theory, produces  different impulse-response functions that are however independent of the autoregressive representation of $\mathbf F_t$.

\section{Data Generating Process for the Simulations}\label{sec_dgp}

The simulation results of Section \ref{ECM}   are obtained using  the  following specification of  \eqref{faziocoglione}:
\beq\label{boschibona} \mathbf A(L) \mathbf F_t=
\mbf A^*(L)  (1-L) \mbf F_t + \bm \alpha \bm \beta' \mbf F_{t-1} =\mathbf C(0) \mathbf u_t= \mbf G \mbf H \mbf u_t,
\eeq
where $r=4$, $q=3$,  $c=3$, the degree of $\mathbf A(L)$ is $2$, so that
the degree of $\mathbf A^*(L)$ is $1$. 
$\mathbf A(L)$ is generated using  the   factorization
 $$\mbf A(L)=\mathbfcal{U}(L) \mathbfcal{M}(L) \mathbfcal{V}(L),$$
 where where $\mathbfcal{U}(L)$ and $\mathbfcal{V}(L)$ are $r\times r$  matrix  polynomials with all  their roots outside the unit circle, and 
$$\mathbfcal{M}(L)= \bpm (1-L)\mbf I_{r-c} & \mbf 0\\ \mbf 0 & \mbf I_c\epm$$
 \citep[see][]{watson94}.
To get a VAR(2)  we set $\mathbfcal{U}(L)=\mbf I_r-\mathbfcal{U}_1 L$, and $\mathbfcal{V}(L)=\mbf I_r$, and then, by rewriting $\mathbfcal{M}(L)=\mbf I_r- \mathbfcal{M}_1 L$, we get $ \mbf A_1= \mathbfcal{M}_1+\mathbfcal{U}_1$, and $ \mbf A_2=  -\mathbfcal{M}_1  \mathbfcal{U}_1$.

The  data  are then  generated as follows. The diagonal elements of the matrix $\mathbfcal{U}_1$ are drawn from a uniform distribution between $0.5$ and $0.8$, while the extra--diagonal elements from a uniform distribution between $0$ and $0.3$.   $\mathbfcal{U}_1$ is then standardized to ensure that its largest eigenvalue is $0.6$. The matrix $\mbf G$ is generated as in \citet{baing07}. Let $\tilde{\mbf G}$ be a $r \times r$ diagonal matrix of rank $q$ with non-zero entries $\tilde{g}_{ii}$ drawn from the uniform distribution between $0.8$ and $1.2$, and let $\check{\mbf G}$ be a random $r \times r$ orthogonal matrix. Then, $\mbf G$ is equal to the first $q$ columns of the matrix $\check{\mbf G} \tilde{\mbf G}^{1/2}$.   Lastly, the matrix $\mbf H$ is such that  the upper $3\times 3$ submatrix of  $\mathbf G\mathbf  H$ is lower triangular.

Results are based on 1000 replications. The matrices $\mathbfcal U_1$, $\mbf G$ and $\mbf H$ are simulated only once so that the set of impulse responses to be estimated is the same for all replications, whereas the vector $\mbf u_t$ is redrawn from ${\cal N}(\mathbf 0,\mathbf I_4)$ at each replication. 

\end{document}